\documentclass[11pt]{amsart}
\usepackage{amsmath,amssymb,amsthm}
\usepackage{graphicx}
\usepackage{pinlabel}
\usepackage{xcolor}
\usepackage[margin=1.1in]{geometry}
\newtheorem{theorem}{Theorem}[section]
\newtheorem{lemma}[theorem]{Lemma}
\newtheorem{corollary}[theorem]{Corollary}
\newtheorem{proposition}[theorem]{Proposition}

\newtheorem{conjecture}[theorem]{Conjecture}
\newtheorem{question}[theorem]{Question}
\theoremstyle{definition}
\newtheorem{definition}[theorem]{Definition}
\newtheorem{example}[theorem]{Example}
\newtheorem{algorithm}[theorem]{Algorithm}
\newtheorem{remark}[theorem]{Remark}
\newcommand{\tarc}{\mbox{\huge$\frown$}}
\newsavebox\arcarg
\DeclareRobustCommand{\ttarc}[1]{{\sbox\arcarg{$#1$}\mathord{\ooalign{\hidewidth\raisebox{1.35ex}[3.35ex]{\resizebox{0.95\wd\arcarg}{\height}{\tarc}}\hidewidth\cr\usebox\arcarg\cr}}}}
\newcommand{\sarc}{\mbox{\huge$\smile$}}
\DeclareRobustCommand{\ssarc}[1]{{\sbox\arcarg{$#1$}\mathord{\ooalign{\hidewidth\raisebox{1.8ex}[3.8ex]{\resizebox{0.95\wd\arcarg}{\height}{\sarc}}\hidewidth\cr\usebox\arcarg\cr}}}}
\newcommand{\Mod}{\operatorname{Mod}}
\newcommand{\Z}{\mathbb Z}
\newcommand{\ilab}{\fontsize{5}{6}\selectfont}

\title[Free products generated by Dehn twists]{Free products generated by Dehn twists\\ about relaxed curves}

\author{Elif Medeto\u{g}ullar\i}
\address{Department of Mathematics and Science Education, Mathematics Education, TED University, 06420 Ankara, Turkey}
\email{elif.medetogullari@tedu.edu.tr}

\author{Elif Dalyan}
\address{Department of Mathematics, Hitit University, 19030 \c{C}orum, Turkey}
\email{elifdalyan@hitit.edu.tr}

\author{S. \"Oyk\"u Yurtta\c{s}}
\address{Department of Mathematics, Bo\u{g}azi\c{c}i University, 34342 Bebek, Istanbul, Turkey}
\email{saadetoyku.yurttas@bogazici.edu.tr}

\subjclass[2020]{20F36, 57K20}
\keywords{Dehn twist, braid group, free product, relaxed curve, Dynnikov coordinates}
\begin{document}

\begin{abstract}
We give sufficient criteria for Dehn twists about relaxed curves in the $n$-punctured disk to generate free products of free abelian groups in the braid group $B_n$. We characterise the sets of these twists which generate free groups, and show that such a set has at most $\lfloor(n+1)/2\rfloor$ elements. We also show that every free product $\Z^{s_1}*\dots*\Z^{s_r}$ with $r\ge2$ is generated by such twists for some $n$. Unlike Humphries' criterion, our criteria impose no condition on the complementary regions of the curves, and they apply to families to which his criterion does not apply. The proofs use a ping-pong argument in terms of $\pi_1$-train track coordinates. Finally, using the formulae which relate $\pi_1$-train track coordinates to Dynnikov coordinates, we give an algorithm which checks our criteria from Dynnikov coordinates.
\end{abstract}

\maketitle

\section{Introduction}\label{sec:introduction}

For $n\ge3$, let $D_n$ be a standard model of the $n$-punctured disk in the plane, with the $n$ punctures arranged from left to right along the horizontal diameter. A simple closed curve in $D_n$ is \emph{inessential} if it bounds an unpunctured disk or a once-punctured disk, or cobounds an unpunctured annulus with $\partial D_n$. It is \emph{essential} otherwise \cite{FM}. Throughout the paper an essential simple closed curve is simply called a \emph{curve}, and all curves are unoriented. Unless we say otherwise, we consider curves up to isotopy. A \emph{multicurve} in $D_n$ is a nonempty union of finitely many pairwise disjoint curves, some of which may be isotopic. We write $\mathcal S(D_n)$ for the set of multicurves in $D_n$. The mapping class group $\Mod(D_n)$ is the group of isotopy classes of orientation preserving homeomorphisms of $D_n$ which fix the boundary pointwise and permute the punctures; it is isomorphic to the braid group $B_n$ \cite{Birman}. For a curve $a$ we write $t_a$ for the Dehn twist about $a$ \cite{FM}; its direction is fixed in Section~\ref{sec:twistrule}. The pure subgroup of $\Mod(D_n)$, which fixes each puncture, is generated by Dehn twists \cite{FM}. We compose mapping classes from right to left, and write $F_k$ for the free group of rank $k$.

We are interested in the following question. Given finitely many curves in $D_n$, what can be said about the subgroup of $\Mod(D_n)$ generated by their Dehn twists, in terms of how the curves intersect? For two curves $a,b\in\mathcal S(D_n)$ we write $\iota(a,b)$ for their geometric intersection number, that is, the minimal number of intersections of representatives in the isotopy classes of $a$ and $b$. Their twists commute if $\iota(a,b)=0$ and generate a free group of rank two if $\iota(a,b)\ge2$ \cite{Ishida}. For three or more curves the pairwise intersection numbers are not enough. For instance, the three curves in $D_3$ of \cite[Example~2.5]{Humphries} surround the first two, the last two, and the first and the third punctures respectively, and pairwise intersect twice, but their twists generate the pure braid group $P_3\cong F_2\times\Z$ and not $F_3$.

In this paper we work with relaxed curves, which are defined as follows (see Figure~\ref{fig:setup}). Write $H$ for the horizontal diameter of $D_n$. The punctures cut $H$ into $n+1$ intervals $I_0,\dots,I_n$, where $I_s$ lies between the $s$-th and the $(s+1)$-st punctures for $1\le s\le n-1$, and $I_0$ and $I_n$ are the two intervals ending on the boundary of $D_n$. The diameter also cuts $D_n$ into the upper half disk $D^U$ and the lower half disk $D^D$. A curve with geometric intersection number two with $H$ is called \emph{relaxed} \cite{DW}. A relaxed curve surrounds the $i$-th to the $j$-th punctures for some $1\le i<j\le n$ with $j-i\le n-2$, and we denote it by $c_{i,j}$. It crosses $H$ once in $I_{i-1}$ and once in $I_j$, and we call $I_{i-1}$ and $I_j$ the \emph{endpoint intervals} of $c_{i,j}$. For instance, $D_3$ contains only two relaxed curves, $c_{1,2}$ and $c_{2,3}$; the curve which surrounds the first and the third punctures is not relaxed. We write $t_{i,j}$ for the Dehn twist $t_{c_{i,j}}$. We say that two relaxed curves $c_{i,j}$ and $c_{k,\ell}$ \emph{alternate} if their endpoint intervals alternate along $H$, that is, if $ i<k\le j<\ell$ or $k<i\le \ell<j$.

\begin{figure}[h!]
\centering
\labellist
\hair 2pt
\pinlabel {{\scriptsize alternating\ (a\,b\,a\,b)}} at 60 20
\pinlabel {{\scriptsize nested\ (a\,b\,b\,a)}} at 215 20
\pinlabel {{\scriptsize side by side\ (a\,a\,b\,b)}} at 377 20
\pinlabel {{\ilab $I_{0}$}} at 12 74
\pinlabel {{\ilab $I_{1}$}} at 39 74
\pinlabel {{\ilab $I_{2}$}} at 63 74
\pinlabel {{\ilab $I_{3}$}} at 86 74
\pinlabel {{\ilab $I_{4}$}} at 110 74
\pinlabel {{\ilab $I_{5}$}} at 134 74
\pinlabel {{\ilab $I_{0}$}} at 163 74
\pinlabel {{\ilab $I_{1}$}} at 187 74
\pinlabel {{\ilab $I_{2}$}} at 214 74
\pinlabel {{\ilab $I_{3}$}} at 241 74
\pinlabel {{\ilab $I_{4}$}} at 261 74
\pinlabel {{\ilab $I_{5}$}} at 285 74
\pinlabel {{\ilab $I_{0}$}} at 314 74
\pinlabel {{\ilab $I_{1}$}} at 341 74
\pinlabel {{\ilab $I_{2}$}} at 369 74
\pinlabel {{\ilab $I_{3}$}} at 385 74
\pinlabel {{\ilab $I_{4}$}} at 412 74
\pinlabel {{\ilab $I_{5}$}} at 436 74
\endlabellist
\includegraphics{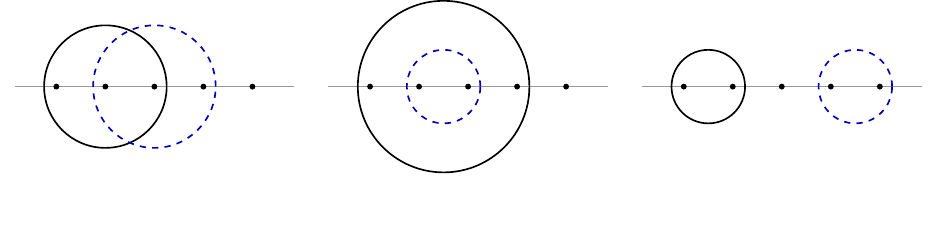}
\caption{Pairs of relaxed curves in $D_5$, with $a$ solid and $b$ dashed: $c_{1,3}$ and $c_{2,4}$ (left), $c_{1,4}$ and $c_{2,3}$ (middle), $c_{1,2}$ and $c_{4,5}$ (right).}
\label{fig:setup}
\end{figure}

\noindent In other words, the four endpoint intervals appear along $H$ in the order $a,b,a,b$; for instance the endpoint intervals of $c_{1,3}$ and $c_{2,4}$ are $I_0,I_1,I_3,I_4$.  Otherwise the order is $a,b,b,a$, and one curve surrounds the other, as for $c_{1,4}$ and $c_{2,3}$; or it is $a,a,b,b$, and the curves lie side by side, as for $c_{1,2}$ and $c_{4,5}$. In these two cases the curves may also have a common endpoint interval, as $c_{1,4}$ and $c_{2,4}$, or $c_{1,2}$ and $c_{3,4}$. Alternating curves intersect twice and curves which do not alternate are disjoint (Remark~\ref{lem:intersection}). So for relaxed curves alternating simply means intersecting. We have $\iota(a,b)=2$ if $a$ and $b$ alternate, and $\iota(a,b)=0$ otherwise.

Our main result gives criteria for the Dehn twists themselves, rather than high powers of them; only its second part replaces some twists by their squares. These criteria involve only endpoint intervals and do not refer to complementary regions. In the proofs we describe curves by their arcs in $D^U$ and $D^D$, that is, by their $\pi_1$-train track coordinates (Section~\ref{sec:chains}). The twist about a relaxed curve winds other curves around it, and creates new crossings with $H$ in its two endpoint intervals. If two curves have a common endpoint interval, the crossings created by their twists may cancel in this interval. For this reason we first assume that no two curves have a common endpoint interval. When two curves do have a common endpoint interval, we square one of their twists. We explain this after Definition~\ref{def:witness}.

\begin{theorem}\label{thm:main}
Let $C$ be a finite set of relaxed curves in $D_n$. Suppose that $C$ is the disjoint union of nonempty subsets $P_1,\dots,P_r$, $r\ge2$, such that any two curves in the same subset are disjoint, and any two curves in different subsets alternate. If no two curves of $C$ have a common endpoint interval, then
\[
\langle t_c : c\in C\rangle\cong\Z^{|P_1|}*\dots*\Z^{|P_r|},
\]
where the factor $\Z^{|P_q|}$ is generated by the twists about the curves of $P_q$. Furthermore, suppose that no curve of $C$ has a common endpoint interval with two other curves of $C$. If, for each pair of curves with a common endpoint interval, we replace one of the two twists by its square, then the same conclusion holds.
\end{theorem}

In the free group case we have the following.

\begin{corollary}\label{cor:free}
Let $c_1,\dots,c_k$, $k\ge2$, be distinct relaxed curves in $D_n$. Then $\langle t_{c_1},\dots,t_{c_k}\rangle$ is a free group if and only if the curves pairwise alternate, and in this case it is freely generated by $t_{c_1},\dots,t_{c_k}$. Furthermore, the largest number of pairwise alternating relaxed curves in $D_n$ is $\lfloor(n+1)/2\rfloor$.
\end{corollary}

Humphries gave a criterion in terms of the complementary regions of the curves. It requires that no component of the complement of the union of the curves is a disk. Following Humphries, a partition $S_1,\dots,S_m$ of a set of distinct curves is called \emph{complete} if curves in the same subset have geometric intersection number zero and curves in different subsets have geometric intersection number at least two.

\begin{theorem}[{\cite[Theorem~2.1]{Humphries}}]\label{thm:humphries}
Let $S=\{c_1,\dots,c_k\}$ be a set of distinct curves on an orientable surface $F$, with or without boundary, let $S_1,\dots,S_m$ be a complete partition of $S$ and let $n_i=|S_i|$. Suppose that no component of $F\setminus\bigcup c_i$ is a disk. Then $\langle t_{c_1},\dots,t_{c_k}\rangle\cong\Z^{n_1}*\dots*\Z^{n_m}$.
\end{theorem}

In \cite{Humphries} the homeomorphisms need not fix the boundary pointwise. So when $F=D_n$, the mapping class group of Theorem~\ref{thm:humphries} is the quotient of $\Mod(D_n)$ by its centre, which is generated by the full twist, that is, the Dehn twist about a curve parallel to the boundary \cite{FM}. The twists in each $S_i$ also commute in $\Mod(D_n)$, so there is a homomorphism from $\Z^{n_1}*\dots*\Z^{n_m}$ onto $\langle t_{c_1},\dots,t_{c_k}\rangle\subset\Mod(D_n)$. Its composition with the quotient map is injective by Theorem~\ref{thm:humphries}, so it is injective. Hence Theorem~\ref{thm:humphries} also holds in $\Mod(D_n)$.

The hypothesis on complementary regions depends on how the curves are drawn, since isotopic curves may have different complementary regions, and it cannot be dropped. For instance, the three curves in $D_3$ mentioned above pairwise intersect twice, so the partition into singletons is complete, but their twists do not generate $F_3$; hence no curves isotopic to these three satisfy the hypothesis. Our results do not need this hypothesis. Instead we only consider relaxed curves, and our conditions are on their endpoint intervals. Since alternating curves intersect twice, the partition $P_1,\dots,P_r$ of Theorem~\ref{thm:main} is complete in the sense of Humphries. The families of Examples~\ref{ex:D6} and~\ref{ex:D7} show that our criteria give free products of free abelian groups in cases where Theorem~\ref{thm:humphries} does not apply. For the curves $c_{1,5},c_{2,4},c_{3,6},c_{4,7}$ in $D_7$, the first part of Theorem~\ref{thm:main} gives $\Z^2*\Z*\Z$ using the Dehn twists themselves. For the curves $c_{1,4},c_{2,4},c_{3,5},c_{4,6}$ in $D_6$, the second part gives the same free product after replacing $t_{1,4}$ by its square. Neither family admits representatives with finitely many pairwise transverse intersections satisfying the complementary region hypothesis of Theorem~\ref{thm:humphries}; see Remark~\ref{remark:multiple-points}.

The free group case is different: for $c_{1,3},c_{2,4},c_{3,5}$ in $D_5$, the representatives in Figure~\ref{fig:F3} have two complementary triangles without punctures, but the representatives through two common points in Figure~\ref{fig:pencil} have no unpunctured complementary disk, so the ``if'' part of Corollary~\ref{cor:free} also follows from Theorem~\ref{thm:humphries}, applied to these representatives (Proposition~\ref{prop:humphries-free}). We discuss this in Section~\ref{sec:disks}.

Free groups generated by Dehn twists have also been studied by other methods. Leininger and Margalit \cite{LM} showed that any two elements of the pure braid group either commute or generate a free group. For two Dehn twists in a punctured disk this was known earlier \cite{Ishida,HT}. In \cite{DMYA} we gave a proof using Dynnikov coordinates for the two intersecting relaxed curves in $D_3$. For three or more twists, Humphries \cite[Theorem~2.8]{Humphries} and Hamidi-Tehrani \cite[Theorem~7.2]{HT} gave criteria in terms of pairwise intersection numbers. Hamidi-Tehrani required $\iota(a_i,a_j)\iota(a_\ell,a_j)/\iota(a_i,a_\ell)\ge6$ for all distinct $i,j,\ell$. For pairwise alternating relaxed curves this ratio is two, so this criterion does not apply. Since these curves have pairwise intersection number two, none of the numerical conditions in \cite[Theorem~2.8]{Humphries} is satisfied either. Results on powers of twists were given in \cite{Funar,Koberda,Runnels,Seo}. Seo \cite[Corollary~1.3]{Seo} showed that if all pairwise intersection numbers of a family of curves are zero or a fixed number $N\ge2$, then the sixth powers of the twists generate a right-angled Artin group. Seo considers surfaces with punctures and without boundary. Replacing the boundary of $D_n$ by a puncture and arguing as after Theorem~\ref{thm:humphries}, we see that this applies to relaxed curves, whose intersection numbers are zero or two. Funar \cite[Theorem~1.10]{Funar} showed that the $D$-th powers of the twists, $D\ge2$, generate a right-angled Artin group for collections of curves which he calls sparse. In Theorem~\ref{thm:main} we use the twists themselves, and squares only for pairs of curves with a common endpoint interval.

Our results are related to the Generalized Tits Conjecture of Jankiewicz and Schreve \cite{JS}. They considered a right-angled Artin group whose generators correspond to the full twists $\Delta^2_{i,j}$ in $B_n$, which are the twists $t_{i,j}$ (Section~\ref{sec:words}), and the map $\Phi_N$ which sends each generator to the $N$-th power of the corresponding full twist. The conjecture states that $\Phi_N$ is injective for some $N$, and for braid groups this follows from \cite{Koberda}. However, $\Phi_1$ is not injective for $n\ge4$ \cite[Example~5.8]{JS}, because of the lantern relation. The relation involves the twists about the five relaxed curves $c_{1,2},c_{2,3},c_{3,4},c_{1,3},c_{2,4}$, whose commutation graph is a cycle of length five. So the twists about relaxed curves do not always generate the expected group, and Theorem~\ref{thm:main} gives families for which they do. Its first part proves injectivity of the restriction of $\Phi_1$ to the subgroup associated with our family, when no two curves have a common endpoint interval and the commutation graph is a disjoint union of complete graphs. The five curves above do not admit a partition as in Theorem~\ref{thm:main}, so our criteria do not apply to them; see also Remark~\ref{rem:lantern}. In order to explain the idea of the proof of Theorem~\ref{thm:main}, which is a ping-pong argument, we need some preliminary definitions. We describe a curve in minimal position with respect to $H$ by its upper and lower arcs, denoted $\ttarc{I_sI_t}$ and $\ssarc{I_sI_t}$ \cite{AY22}. Their numbers are the $\pi_1$-train track coordinates defined in Section~\ref{sec:chains}. The sets used in the ping-pong argument are defined by these arcs; we call them witness sets.

\begin{definition}\label{def:witness}
Let $c=c_{i,j}$ be a relaxed curve.
\begin{enumerate}
\item The set of curves $\gamma\in\mathcal S(D_n)$ which, like $c$ itself, have both an upper arc $\ttarc{I_{i-1}I_j}$ and a lower arc $\ssarc{I_{i-1}I_j}$ is called the \emph{witness set} of $c$, and we denote it by $Q_c=Q_{i,j}$.
\item Let $c_{k,\ell}$ be a relaxed curve which alternates with $c$. Then one of $I_{k-1}$ and $I_\ell$ lies inside $c$ and the other lies outside $c$, and an arc $\ttarc{I_{k-1}I_\ell}$ or $\ssarc{I_{k-1}I_\ell}$ of a curve is called a \emph{witness arc} for $c$.
\end{enumerate}
\end{definition}

If $c'$ alternates with $c$, then every curve in $Q_{c'}$ has witness arcs for $c$, namely its upper and lower arcs of $c'$. Twisting along $c$ winds these witness arcs around $c$, producing both an upper and a lower arc of $c$. Hence every nonzero power of $t_c$ sends $Q_{c'}$ into $Q_c$ (Corollary~\ref{cor:alt}); this is why we call these arcs witness arcs. Twists about curves which are disjoint from $c$ and have no endpoint interval in common with $c$ preserve the upper and lower arcs of $c$, so they send $Q_c$ to itself (Corollary~\ref{cor:nonalt}). This gives the first part of Theorem~\ref{thm:main} by the Ping-Pong Lemma.

When two curves of a subset have a common endpoint interval, the argument needs the square in the second part of Theorem~\ref{thm:main}, for the following reason. Let $a$ and $b$ be two curves of a subset with a common endpoint interval, and let $c$ be a curve in another subset, so that $c$ intersects both $a$ and $b$. In Lemma~\ref{lem:paired} we show that $t_a^mt_b^{m'}$ sends $Q_c$ into $Q_a\cup Q_b$ for every $(m,m')\ne(0,0)$, except for $(m,m')=\pm(1,-1)$, that is, except for $t_at_b^{-1}$ and its inverse.

In the image of a witness arc, the crossings created by the two twists can cancel only at their shared endpoint interval, and at most one pair cancels; see Lemmas~\ref{lem:W} and~\ref{lem:paired}. Example~\ref{ex:D6} shows that for $t_at_b^{-1}$ the image of $Q_c$ need not lie in $Q_a\cup Q_b$. If we replace $t_a$ by $t_a^2$, then in the free abelian group generated by $t_a^2$, $t_b$ and the other twists of the subset, the exponent $m$ of $t_a$ is always even, so $(m,m')\ne\pm(1,-1)$. We do not know whether the square is necessary (Question~\ref{q:square}), but it is not necessary when every curve surrounds exactly two punctures (Remark~\ref{rem:tits}). The method also gives computations in Dynnikov coordinates. The $\pi_1$-train track coordinates are explicit functions of the Dynnikov coordinates, so whether a curve lies in a witness set can be decided from its Dynnikov coordinates (Remark~\ref{prop:dynnikov}). The action of twists on the coordinate vectors is given by the update rules of \cite{Dynnikov,HY09,HY18}. Finally, Algorithm~\ref{alg:witness} takes the Dynnikov coordinates of a collection of curves and first checks whether they are relaxed. If they are, it decides whether their Dehn twists generate a free group and checks the hypotheses of Theorem~\ref{thm:main}. When these hypotheses hold, it gives the corresponding free product of free abelian groups, with some twists replaced by their squares when required.

This paper is organised as follows. In Section~\ref{sec:words} we describe curves by their $\pi_1$-train tracks. In Section~\ref{sec:twists} we describe the action of a Dehn twist about a relaxed curve on chains of arcs and prove the witness lemmas. In Section~\ref{sec:A} we prove Theorem~\ref{thm:main} and Corollary~\ref{cor:free}. Section~\ref{sec:examples} contains examples, and Section~\ref{sec:disks} compares our results with Theorem~\ref{thm:humphries} and lists some questions. In Section~\ref{sec:algorithm} we give the algorithm. In Appendix~\ref{sec:dynnikov-coordinates} we recall the Dynnikov coordinate system \cite{Dynnikov,HY09} and give the transition formulae between $\pi_1$-train track coordinates and Dynnikov coordinates.

\section{Preliminaries}\label{sec:words}

We use the notation of the introduction, and we regard $n$ as fixed throughout. Since the punctures lie on $H$, the half disks $D^U$ and $D^D$ contain no punctures. For $1\le s\le n-1$ we write $\sigma_s$ for the braid generator which interchanges the $s$-th and the $(s+1)$-st punctures counterclockwise \cite{HY09}. The twists $t_{i,j}$ have a simple interpretation in terms of these generators. Let
\[
\Delta_{i,j}=(\sigma_i\sigma_{i+1}\cdots\sigma_{j-1})(\sigma_i\cdots\sigma_{j-2})\cdots(\sigma_i\sigma_{i+1})\,\sigma_i
\]
be the Garside element of the braids on the strands $i,\dots,j$ \cite{Garside}. Then
\[
\Delta^2_{i,j}=(\sigma_i\sigma_{i+1}\cdots\sigma_{j-1})^{j-i+1}
\]
is the full twist of these strands, that is, $t_{i,j}=\Delta^2_{i,j}$ \cite{CGGW,FM}. For example, $t_{i,i+1}=\sigma_i^2$.

\subsection{$\pi_1$-train track coordinates}\label{sec:chains}
Given a curve $\gamma\in\mathcal S(D_n)$, we can always find a \emph{minimal} representative of $\gamma$ with respect to $H$, that is, a representative which is transverse to $H$ and has the minimal number of intersections with $H$. Unless otherwise stated, we always take a minimal representative of $\gamma$, and we denote it by $\gamma$ as well. Representatives which are not minimal only appear as images of minimal representatives under Dehn twists, and we make them minimal by Move~1 below. Since $H$ cuts $\gamma$ into arcs, each of which lies in $D^U$ or in $D^D$ and has its endpoints on $H$, we describe $\gamma$ by these arcs, with the notation of \cite{AY22}. We always draw a minimal representative so that each of its arcs is a semicircle centred on $H$. In particular, a relaxed curve consists of one upper and one lower semicircle, as in Figure~\ref{fig:setup}. Similar descriptions of curves are the curve diagrams of \cite{Cumplido} and \cite{FGRRW}.

\begin{definition}\label{def:chain}
Let $\gamma\in\mathcal S(D_n)$ be a curve.
\begin{enumerate}
\item An arc of $\gamma$ in $D^U$ with endpoints in $I_s$ and $I_t$ is called an \emph{upper arc}, and an arc of $\gamma$ in $D^D$ with endpoints in $I_s$ and $I_t$ is called a \emph{lower arc}. Following \cite{AY22}, we denote them by $\ttarc{I_sI_t}$ and by $\ssarc{I_sI_t}$.
\item We call the points of $\gamma\cap H$ the \emph{crossings} of $\gamma$. If $\gamma$ crosses $H$ in the intervals $I_{s_1},\dots,I_{s_N}$, in this order, then the \emph{chain} of $\gamma$ is
\[
w(\gamma)=\ttarc{I_{s_1}I_{s_2}}\ssarc{I_{s_2}I_{s_3}}\ttarc{I_{s_3}I_{s_4}}\cdots\ssarc{I_{s_N}I_{s_1}}.
\]
It is defined up to cyclic permutation and reversal.
\item For $0\le s<t\le n$ we write $\mu_\gamma(\ttarc{I_sI_t})$ and $\mu_\gamma(\ssarc{I_sI_t})$ for the number of upper arcs and of lower arcs of $\gamma$ joining $I_s$ and $I_t$.
\item We say that the chain $w(\gamma)$ is \emph{reduced} if it contains no arc of the form $\ttarc{I_sI_s}$ or $\ssarc{I_sI_s}$.
\end{enumerate}
\end{definition}

Each symbol has the shape of the arc it denotes, and since an arc does not depend on the order of its endpoints, $\ttarc{I_sI_t}=\ttarc{I_tI_s}$ and $\ssarc{I_sI_t}=\ssarc{I_tI_s}$. For $i\le j$ the upper arc $\ttarc{I_{i-1}I_j}$ passes over the $i$-th to the $j$-th punctures, and the lower arc $\ssarc{I_{i-1}I_j}$ passes under them. The arcs of $\gamma$ are alternately upper and lower, so $N$ is even. For example, the relaxed curve $c_{1,5}$ crosses $I_0$, runs over the first five punctures, crosses $I_5$ and returns under the same punctures, so $w(c_{1,5})=\ttarc{I_0I_5}\ssarc{I_5I_0}$. In general $ w(c_{i,j})=\ttarc{I_{i-1}I_j}\ssarc{I_jI_{i-1}}.$ We call $\ttarc{I_{i-1}I_j}$ and $\ssarc{I_{i-1}I_j}$ the \emph{upper arc} and the \emph{lower arc} of $c_{i,j}$. We say that a crossing in $I_s$ is \emph{upward} if $\gamma$ passes there from a lower arc to an upper arc, as in $\ssarc{I_rI_s}\ttarc{I_sI_t}$, and \emph{downward} otherwise. Let $\gamma$ be minimal. Then two arcs of $\gamma$ in the same half disk with the same endpoint intervals are parallel. So the numbers $\mu_\gamma(\ttarc{I_sI_t})$ and $\mu_\gamma(\ssarc{I_sI_t})$ count the number of parallel arcs of $\gamma$ in $D^U$ and in $D^D$ joining $I_s$ and $I_t$. These numbers are the \emph{$\pi_1$-train track coordinates} of $\gamma$ \cite{YG19}. Let $\tau$ be the graph obtained from $\gamma$ by collapsing the crossings of $\gamma$ in each interval $I_s$ to one point, called a \emph{switch}, and the parallel arcs of $\gamma$ to one edge, called a \emph{branch}. At each switch the branches meet with a common tangent line, which is transverse to $H$. We call $\tau$ the \emph{$\pi_1$-train track} of $\gamma$ \cite{AY22,YG19}; such train tracks were introduced by Birman and Series \cite{BS}. We say that $\tau$ \emph{carries} $\gamma$, since $\gamma$ is isotopic to a curve in a small neighbourhood of $\tau$ which runs along the branches of $\tau$. The \emph{measure} of a branch is the number of arcs of $\gamma$ parallel to it. The branches are denoted in the same way as the arcs (Figure~\ref{fig:pi1tt}). A \emph{train path} in $\tau$ passes at each switch from an upper branch to a lower branch or conversely. Thus $w(\gamma)$ gives a closed train path in $\tau$, traversing each branch as many times as its measure. For $c_{i,j}$, the two branches $\ttarc{I_{i-1}I_j}$ and $\ssarc{I_{i-1}I_j}$ have measure one, and all other measures are zero. Figure~\ref{fig:pi1tt} gives another example.

\begin{figure}[ht]
\centering
\labellist
\hair 2pt
\pinlabel {{\ilab $I_{0}$}} at 10 85
\pinlabel {{\ilab $I_{1}$}} at 40 85
\pinlabel {{\ilab $I_{2}$}} at 63 85
\pinlabel {{\ilab $I_{3}$}} at 87 85
\pinlabel {{\ilab $I_{4}$}} at 110 85
\pinlabel {{\ilab $I_{5}$}} at 130 85
\pinlabel {{\ilab $I_{6}$}} at 153 85
\pinlabel {{\ilab $I_{7}$}} at 180 85
\pinlabel {{\ilab $I_{0}$}} at 236 85
\pinlabel {{\ilab $I_{1}$}} at 259 85
\pinlabel {{\ilab $I_{2}$}} at 283 85
\pinlabel {{\ilab $I_{3}$}} at 306 85
\pinlabel {{\ilab $I_{4}$}} at 329 85
\pinlabel {{\ilab $I_{5}$}} at 349 85
\pinlabel {{\ilab $I_{6}$}} at 373 85
\pinlabel {{\ilab $I_{7}$}} at 399 85
\pinlabel {{\tiny $\ttarc{I_0I_2}$}} at 259 118
\pinlabel {{\tiny $\ttarc{I_2I_6}$}} at 329 153
\pinlabel {{\tiny $\ttarc{I_5I_6}$}} at 364 112
\pinlabel {{\tiny $\ssarc{I_0I_6}$}} at 312 8
\pinlabel {{\tiny $\ssarc{I_2I_5}$}} at 318 63
\pinlabel {{\tiny $\ssarc{I_2I_6}$}} at 281 42
\endlabellist
\includegraphics{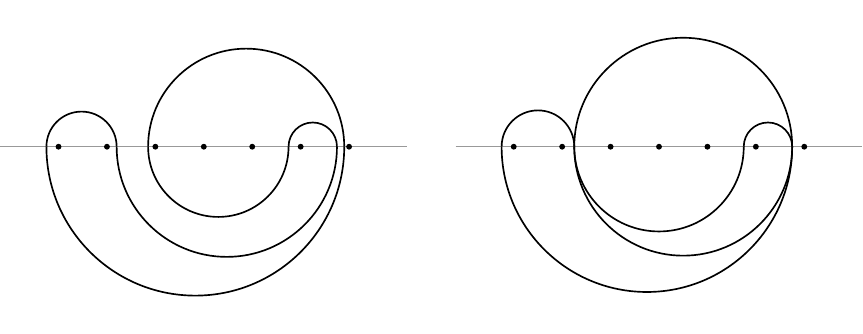}
\caption{The curve $\gamma=t_{3,6}(c_{1,5})$ and its carrying train track. Every displayed branch has measure one; all other measures are zero. The same curve is used in Examples~\ref{ex:twist} and~\ref{ex:dynnikov}.}
\label{fig:pi1tt}
\end{figure}
If $w(\gamma)$ contains an arc $\ttarc{I_sI_s}$ or $\ssarc{I_sI_s}$, then this arc and a segment of $I_s$ bound a bigon $E$ between $\gamma$ and $H$, which contains no puncture. Every arc of $\gamma\cap E$ has both endpoints on this segment, so we may choose the arc innermost. Pushing it across $E$ removes two crossings with $H$ and creates none, and in the chain it replaces
\[
\ssarc{I_rI_s}\ttarc{I_sI_s}\ssarc{I_sI_t}\ \text{ by }\ \ssarc{I_rI_t},\qquad\text{or}\qquad \ttarc{I_rI_s}\ssarc{I_sI_s}\ttarc{I_sI_t}\ \text{ by }\ \ttarc{I_rI_t}.
\]
This is Move~1 (removing bigons) \cite{AY22,Cumplido}. In this paper bigons arise when we apply a Dehn twist: the image of a minimal curve need not be minimal, and we remove its bigons with $H$ by Move~1 (Proposition~\ref{prop:twist}). The next lemma shows that Move~1 is enough to make a curve minimal. It is the closed curve version of a result in \cite{FGRRW}; see also \cite{Cumplido}.

\begin{lemma}\label{lem:reduced}
Let $\gamma\in\mathcal S(D_n)$ be a curve.
\begin{enumerate}
\item $\gamma$ is minimal if and only if $w(\gamma)$ is reduced.
\item If $\gamma$ and $\gamma'$ are minimal and isotopic, then $w(\gamma)$ and $w(\gamma')$ agree up to cyclic permutation and reversal.
\end{enumerate}
\end{lemma}

\begin{proof}
Choose the usual spine $G$ of $D_n$ with two vertices $v_U\in D^U$ and $v_D\in D^D$ and $n+1$ edges $e_0,\dots,e_n$, where $e_s$ joins the vertices through $I_s$. A regular neighbourhood of $G$ is a disk with $n$ holes, so its inclusion in $D_n$ is a homotopy equivalence. Under the resulting description by edge paths, an arc in either half disk joins successive crossings without adding an edge, and a crossing through $I_s$ records $e_s$ with the appropriate direction. So a curve with chain $\ttarc{I_{s_1}I_{s_2}}\ssarc{I_{s_2}I_{s_3}}\cdots\ssarc{I_{s_N}I_{s_1}}$ goes to the closed edge path $e_{s_1}e_{s_2}\cdots e_{s_N}$ in $G$, each edge traversed in the direction of the curve. An arc $\ttarc{I_sI_s}$ or $\ssarc{I_sI_s}$ of $\gamma$ gives $e_s$ followed by its inverse in this edge path, so Move~1 on $\gamma$ is the cancellation of an edge against its inverse. Hence $w(\gamma)$ is reduced if and only if the edge path is cyclically reduced. Isotopic curves are freely homotopic in $D_n$, so their edge paths are freely homotopic in $G$. In a graph, every free homotopy class of closed edge paths contains a cyclically reduced edge path, which is unique up to cyclic permutation and reversal, and it is obtained from any edge path in the class by cancelling edges against their inverses \cite{LS}. So the cyclically reduced edge path is the shortest one in its class. Hence a curve with reduced chain crosses $H$ at most as often as any isotopic curve, that is, it is minimal. Conversely, if $w(\gamma)$ is not reduced, then $\gamma$ has an arc $\ttarc{I_sI_s}$ or $\ssarc{I_sI_s}$, which bounds a bigon with $H$, and Move~1 removes it. So $\gamma$ is not minimal. This proves (1). For (2), two isotopic minimal curves have reduced chains by (1), so their edge paths are cyclically reduced and freely homotopic, and hence they agree up to cyclic permutation and reversal.
\end{proof}
We use Lemma~\ref{lem:reduced} as follows. Every curve is isotopic to a minimal curve, which is obtained by applying Move~1 as long as possible. Two isotopic minimal curves have the same chain, up to cyclic permutation and reversal, and therefore the same measures $\mu_\gamma(\ttarc{I_sI_t})$ and $\mu_\gamma(\ssarc{I_sI_t})$. So these depend only on the isotopy class of $\gamma$. We shall also use chains of paths. If a path has its endpoints on $H$ and is otherwise transverse to $H$, its chain is defined similarly, and applying Move~1 to the arcs between its endpoints gives the \emph{reduced chain} of the path. As in the proof of Lemma~\ref{lem:reduced}, two paths which agree near their endpoints have the same reduced chain if and only if they are homotopic relative to these neighbourhoods.

\begin{remark}\label{rem:witnessdef}
By Definition~\ref{def:witness}, a curve $\gamma$ lies in the witness set $Q_{c_{i,j}}=Q_{i,j}$ of $c_{i,j}$ if and only if $\mu_\gamma(\ttarc{I_{i-1}I_j})>0$ and $\mu_\gamma(\ssarc{I_{i-1}I_j})>0$. In particular $c_{i,j}\in Q_{i,j}$.
\end{remark}

\begin{lemma}\label{lem:disjoint}
If $c_{i,j}$ and $c_{k,\ell}$ alternate, then $Q_{i,j}\cap Q_{k,\ell}=\emptyset$.
\end{lemma}

\begin{proof}
Two disjoint arcs in the closure of $D^U$ with endpoints on $H$ cannot have alternating endpoints. Hence a minimal curve cannot have an upper arc $\ttarc{I_{i-1}I_j}$ and an upper arc $\ttarc{I_{k-1}I_{\ell}}$ at the same time.
\end{proof}

\begin{remark}\label{lem:intersection}
If two relaxed curves $c_{i,j}$ and $c_{k,\ell}$ alternate, then the sets of punctures which they surround overlap and neither contains the other. So the curves must intersect, since disjoint curves surround nested or disjoint sets of punctures. Each of them consists of one upper and one lower arc. In minimal position the two upper arcs meet once, and so do the two lower arcs, since their endpoints alternate. Hence $\iota(c_{i,j},c_{k,\ell})=2$. If they do not alternate, then one of them surrounds the other or they lie side by side, so $\iota(c_{i,j},c_{k,\ell})=0$.
\end{remark}

\begin{remark}\label{prop:dynnikov}
Dynnikov coordinates give a bijection $\rho\colon\mathcal S(D_n)\to\Z^{2n-4}\setminus\{0\}$ \cite{Dynnikov,HY09}. For a multicurve $\mathcal L\in\mathcal S(D_n)$ we write $\mu_{\mathcal L}(\ttarc{I_sI_t})$ and $\mu_{\mathcal L}(\ssarc{I_sI_t})$ for the number of upper arcs and of lower arcs of $\mathcal L$ joining $I_s$ and $I_t$, as for a curve. These measures are expressed in terms of Dynnikov coordinates in \cite{YG19}; see Appendix~\ref{sec:dynnikov-coordinates}, where we also give an example. Hence whether a curve lies in a witness set can be decided from its Dynnikov coordinates.
\end{remark}

\section{Dehn twists and witness arcs}\label{sec:twists}

In Sections~\ref{sec:twistrule}--\ref{sec:witness} we fix a relaxed curve $c=c_{i,j}$, and we write $\kappa=i-1$ and $\nu=j$, so that $I_\kappa$ and $I_\nu$ are the endpoint intervals of $c$, $0\le\kappa<\nu\le n$ and $2\le\nu-\kappa\le n-1$. An interval $I_s$ is, with respect to $c$ (see Figure~\ref{fig:example}(a)),
\begin{itemize}
\item \emph{boundary} if $s\in\{\kappa,\nu\}$, that is, $c$ itself crosses $H$ in $I_s$;
\item \emph{inner} if $\kappa<s<\nu$, that is, $I_s$ lies inside the disk bounded by $c$;
\item \emph{outer} if $s<\kappa$ or $s>\nu$, that is, $I_s$ lies outside this disk.
\end{itemize}
A crossing of a curve with $H$, or an endpoint of an arc, is boundary, inner or outer with respect to $c$ if the interval containing it is. These notions depend on $c$, and when several curves are involved we always say with respect to which curve; for instance, in Section~\ref{sec:B} an interval may be inner with respect to one curve and outer with respect to another. For example, let $c=c_{3,6}$ and $\gamma=c_{1,5}$ in $D_7$, as in Figure~\ref{fig:example}(a). Then $\kappa=2$ and $\nu=6$, the boundary intervals are $I_2,I_6$, the inner intervals are $I_3,I_4,I_5$ and the outer intervals are $I_0,I_1,I_7$. Thus the upper arc $\ttarc{I_0I_5}$ of $\gamma$ has an outer endpoint in $I_0$ and an inner endpoint in $I_5$ with respect to $c$. Observe that an arc crosses $c$ if and only if one of its endpoints is inner and the other is outer or boundary with respect to $c$; here a boundary endpoint is drawn outside the disk bounded by $c$ (Section~\ref{sec:twistrule}). For instance, in the example above both arcs of $\gamma$ cross $c$, while no arc of $c_{1,2}$ crosses $c$. The Dehn twist $t_c$ only changes a curve near $c$. An arc which stays inside $c$, or stays outside $c$, is not affected, while the image under $t_c^m$ of an arc which crosses $c$ goes $|m|$ times around $c$. Proposition~\ref{prop:twist} below gives the details, and Lemma~\ref{lem:arcs} shows that arcs whose endpoints are both inner or both outer survive in the reduced chain.  
\subsection{The twist rule}\label{sec:twistrule}
Let $A$ be an annulus around $c$ such that $A$ meets $H$ in two intervals $J_\kappa$ and $J_\nu$, where $J_\kappa\subset I_\kappa$ and $J_\nu\subset I_\nu$. The twist $t_c$ is supported in $A$, and it turns a path which enters the punctured disk bounded by $c$ once counterclockwise around $c$; see Figure~\ref{fig:dehn}. This is the twist $t_{i,j}=\Delta^2_{i,j}$ of Section~\ref{sec:words}. It is the inverse of the Dehn twist of \cite{FM}, where a path turns left; the subgroups which we consider do not depend on this choice. Since $t_c$ depends only on the isotopy class of $c$, we may take $A$ as thin as we like. Given a curve $\gamma$, first slide its crossings in $I_\kappa$ and $I_\nu$ outside the disk bounded by $c$. We then take $A$ so thin that $A$ contains no crossing of $\gamma$ with $H$, every arc of $\gamma$ which crosses $c$ meets $A$ in a single arc joining its boundary components, and every other arc of $\gamma$ is disjoint from $A$ and hence fixed by $t_c$.

\begin{figure}[ht]
\centering
\labellist
\hair 2pt
\pinlabel {{\scriptsize $c$}} at 80 20
\pinlabel {{\scriptsize $A$}} at 109 73
\pinlabel {{\scriptsize $\delta$}} at 8 48
\pinlabel {{\scriptsize $A$}} at 237 73
\pinlabel {{\scriptsize $t_c(\delta)$}} at 136 48
\endlabellist
\includegraphics{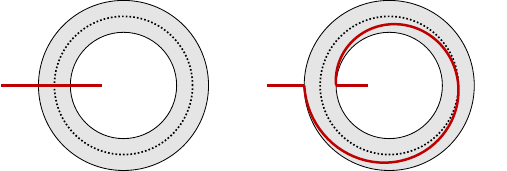}
\caption{The Dehn twist $t_c$ is supported in an annulus $A$ around $c$ (dotted). A path which crosses $A$ turns right as it enters $A$. In particular, a path $\delta$ which crosses $A$ from outside to inside is turned once counterclockwise around $c$.}
\label{fig:dehn}
\end{figure}

Let $\delta$ be an arc in $A$ which joins the outer boundary of $A$ to the inner boundary of $A$ and is disjoint from $H$, as in Figure~\ref{fig:dehn}, and suppose that $\delta$ lies in the upper half disk. Let $m\in\Z$. Then $t_c^m(\delta)$ goes $|m|$ times around $c$ inside $A$, counterclockwise if $m>0$ and clockwise if $m<0$. After removing pairs of consecutive crossings in the same interval, it crosses $H$ exactly $2|m|$ times, alternately in $J_\kappa$ and $J_\nu$. We write $T^m_U(\delta)$ for the sequence of the indices of the intervals of these $2|m|$ crossings, in the order in which $t_c^m(\delta)$ crosses them from the outer boundary of $A$. We define $T^m_D(\delta)$ similarly for an arc $\delta$ in the lower half disk. Any two such arcs $\delta$ in the same half disk give the same sequence after this reduction, so we simply write $T^m_U$ and $T^m_D$, or $T^m_{c,U}$ and $T^m_{c,D}$ when we need to indicate the curve. So $T^0_U$ and $T^0_D$ are empty, and for $m\ge1$
\[
T^m_U=T^{-m}_D=(\kappa,\nu,\kappa,\nu,\dots,\kappa,\nu),\qquad T^m_D=T^{-m}_U=(\nu,\kappa,\nu,\kappa,\dots,\nu,\kappa),
\]
with $m$ pairs each. Reading $T^m_U$ backwards gives $T^{-m}_U$, and similarly for $T^m_D$. For example, for $c=c_{3,6}$ we have $\kappa=2$ and $\nu=6$, so $T^1_U=(2,6)$, $T^2_U=(2,6,2,6)$ and $T^{-1}_U=(6,2)$. Indeed, a path which starts in the upper half disk and goes once counterclockwise around $c$ crosses $H$ first in $I_2$ and then in $I_6$; this can be seen in Figure~\ref{fig:example}(b).

\begin{figure}[ht]
\centering
\labellist
\hair 2pt
\pinlabel {{\ilab $I_{0}$}} at 16 105
\pinlabel {{\ilab $I_{1}$}} at 40 105
\pinlabel {{\ilab $I_{2}$}} at 64 105
\pinlabel {{\ilab $I_{3}$}} at 89 105
\pinlabel {{\ilab $I_{4}$}} at 113 105
\pinlabel {{\ilab $I_{5}$}} at 137 105
\pinlabel {{\ilab $I_{6}$}} at 161 105
\pinlabel {{\ilab $I_{7}$}} at 185 105
\pinlabel {{\scriptsize $D^U$}} at 185 132
\pinlabel {{\scriptsize $D^D$}} at 185 88
\pinlabel {{\tiny $\ttarc{I_0I_5}$}} at 22 153
\pinlabel {{\tiny $\ssarc{I_0I_5}$}} at 22 67
\pinlabel {{\scriptsize \textcolor{black!60}{$c_{3,6}$}}} at 149 145
\pinlabel {{\tiny outer}} [l] at 29 20
\pinlabel {{\tiny boundary}} [l] at 83 20
\pinlabel {{\tiny inner}} [l] at 148 20
\pinlabel {(a)} at 101 5
\pinlabel {{\ilab $I_{0}$}} at 243 105
\pinlabel {{\ilab $I_{1}$}} at 268 105
\pinlabel {{\ilab $I_{2}$}} at 292 105
\pinlabel {{\ilab $I_{3}$}} at 316 105
\pinlabel {{\ilab $I_{4}$}} at 340 105
\pinlabel {{\ilab $I_{5}$}} at 361 105
\pinlabel {{\ilab $I_{6}$}} at 387 105
\pinlabel {{\ilab $I_{7}$}} at 413 105
\pinlabel {{\tiny \textcolor{red!75!black}{$\ttarc{I_2I_6}$}}} at 349 169
\pinlabel {{\tiny \textcolor{blue!70!black}{$\ssarc{I_2I_6}$}}} at 293 62
\pinlabel {(b)} at 328 5
\endlabellist
\includegraphics{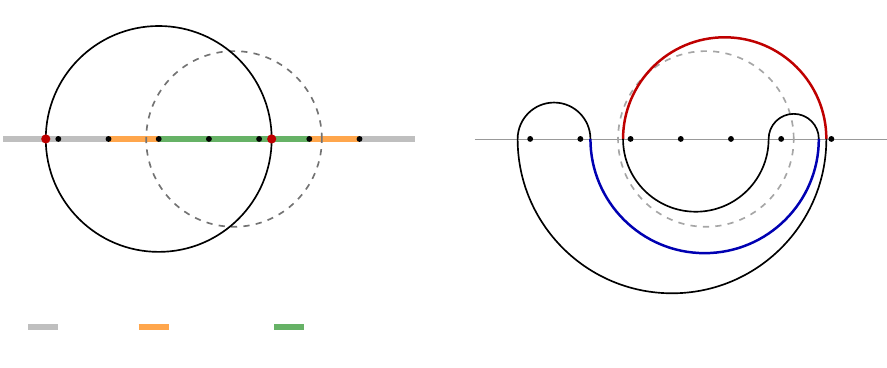}
\caption{(a) The curve $c_{1,5}$ and, dashed, the curve $c=c_{3,6}$ in $D_7$. The intervals of $H$ are marked according to their position with respect to $c$. The intervals $I_2$ and $I_6$ are boundary, $I_3,I_4,I_5$ are inner and $I_0,I_1,I_7$ are outer with respect to $c$. The curve $c_{1,5}$ consists of the upper arc $\ttarc{I_0I_5}$ and the lower arc $\ssarc{I_0I_5}$, so $w(c_{1,5})=\ttarc{I_0I_5}\ssarc{I_5I_0}$; with respect to $c$ both arcs have an outer endpoint in $I_0$ and an inner endpoint in $I_5$, so they cross $c$. (b) The curve $t_c(c_{1,5})$, with chain $\ttarc{I_0I_2}\ssarc{I_2I_6}\ttarc{I_6I_5}\ssarc{I_5I_2}\ttarc{I_2I_6}\ssarc{I_6I_0}$ (Example~\ref{ex:twist}). Its arcs $\ttarc{I_2I_6}$ (red) and $\ssarc{I_2I_6}$ (blue), the upper and lower arcs of $c$, show that it lies in $Q_{3,6}$; under $t_c$ both arcs of $c_{1,5}$ go once around $c$.}
\label{fig:example}
\end{figure}

\begin{proposition}[Twist rule]\label{prop:twist}
Let $c=c_{i,j}$, let $\gamma\in\mathcal S(D_n)$ be a curve and let $m\ge1$. A chain of $t_c^m(\gamma)$ is obtained from $w(\gamma)$ by applying the following rules to every pair $p,q$ with $i-1<p<j$ and either $q\le i-1$ or $q\ge j$.
\begin{enumerate}
\item Every upper arc $\ttarc{I_qI_p}$ of $\gamma$ which crosses $c$ is replaced by
\[
\ttarc{I_qI_{\kappa}}\bigl(\ssarc{I_{\kappa}I_{\nu}}\ttarc{I_{\nu}I_{\kappa}}\bigr)^{m-1}\ssarc{I_{\kappa}I_{\nu}}\ttarc{I_{\nu}I_p}.
\]
\item Every lower arc $\ssarc{I_qI_p}$ of $\gamma$ which crosses $c$ is replaced by
\[
\ssarc{I_qI_{\nu}}\bigl(\ttarc{I_{\nu}I_{\kappa}}\ssarc{I_{\kappa}I_{\nu}}\bigr)^{m-1}\ttarc{I_{\nu}I_{\kappa}}\ssarc{I_{\kappa}I_p}.
\]
\item All other arcs of $\gamma$ are unchanged.
\end{enumerate}
A chain of $t_c^{-m}(\gamma)$ is obtained by the same rules, with $\kappa$ and $\nu$ exchanged. In both cases the reduced chain is obtained from this chain by Move~1.
\end{proposition}
In other words, the image of an upper arc crossing $c$ follows the arc from $I_q$, goes $|m|$ times around $c$, crossing $H$ in the intervals of $T^{\pm m}_U$, and then continues to $I_p$; the image of a lower arc does the same with $T^{\pm m}_D$. Here $q$ may be $\kappa$ or $\nu$, that is, the arc may have a boundary endpoint. In this case the first arc given by the rules may be of the form $\ttarc{I_qI_q}$ or $\ssarc{I_qI_q}$, and it is removed by Move~1, as in Example~\ref{ex:D6}.

\begin{proof}
Sliding the crossings of $\gamma$ along $I_\kappa$ and $I_\nu$, we may assume that all crossings in these two intervals lie outside the punctured disk bounded by $c$; this changes neither the chain nor the isotopy class. Let $\ttarc{I_qI_p}$ or $\ssarc{I_qI_p}$ be an arc of $\gamma$ which crosses $c$. We follow it from $I_q$. When it reaches $A$, its image under $t_c^m$ goes $m$ times counterclockwise around $c$ and then continues to $I_p$. On the way around it crosses $H$ alternately in $J_\kappa$ and $J_\nu$. An upper arc enters $A$ in $D^U$, so its first crossing is in $J_\kappa$; a lower arc enters $A$ in $D^D$, so its first crossing is in $J_\nu$. Hence the crossings are in the order given by $T^m_U$ for an upper arc and by $T^m_D$ for a lower arc (Figure~\ref{fig:rule}). The image under $t_c^{-m}$ goes $m$ times clockwise around $c$, so $\kappa$ and $\nu$ are exchanged. The last statement is Lemma~\ref{lem:reduced}.
\end{proof}

\begin{figure}[ht]
\centering
\labellist
\hair 2pt
\pinlabel {{\ilab $I_{0}$}} at 14 71
\pinlabel {{\ilab $I_{1}$}} at 37 71
\pinlabel {{\ilab $I_{2}$}} at 63 71
\pinlabel {{\ilab $I_{3}$}} at 86 71
\pinlabel {{\ilab $I_{4}$}} at 108 71
\pinlabel {{\ilab $I_{5}$}} at 134 71
\pinlabel {{\ilab $I_{6}$}} at 154 71
\pinlabel {{\scriptsize \textcolor{black!60}{$c$}}} at 119 108
\pinlabel {(a)} at 86 5
\pinlabel {{\ilab $I_{0}$}} at 232 71
\pinlabel {{\ilab $I_{1}$}} at 256 71
\pinlabel {{\ilab $I_{2}$}} at 281 71
\pinlabel {{\ilab $I_{3}$}} at 304 71
\pinlabel {{\ilab $I_{4}$}} at 327 71
\pinlabel {{\ilab $I_{5}$}} at 352 71
\pinlabel {{\ilab $I_{6}$}} at 372 71
\pinlabel {{\scriptsize \textcolor{black!60}{$c$}}} at 338 108
\pinlabel {(b)} at 304 5
\endlabellist
\includegraphics{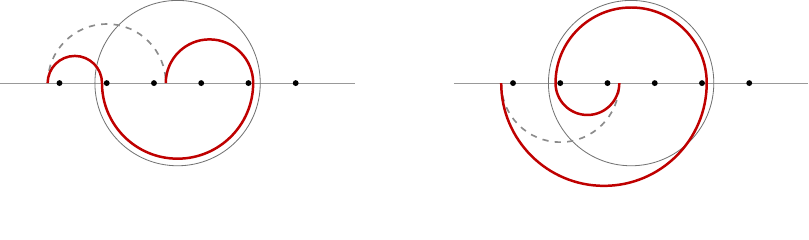}
\caption{The twist rule for $c=c_{2,5}$ in $D_6$, with endpoint intervals $I_1$, $I_5$. The arcs drawn dashed have their outer endpoint in $I_0$ and their inner endpoint in $I_3$, so they cross $c$; their images under $t_c$ are drawn in red. (a) The upper arc $\ttarc{I_0I_3}$ goes once counterclockwise around $c$, crossing $H$ first in $I_1$ and then in $I_5$, and becomes $\ttarc{I_0I_1}\ssarc{I_1I_5}\ttarc{I_5I_3}$. (b) The lower arc $\ssarc{I_0I_3}$ goes once counterclockwise around $c$, crossing $H$ first in $I_5$ and then in $I_1$, and becomes $\ssarc{I_0I_5}\ttarc{I_5I_1}\ssarc{I_1I_3}$. An arc which does not cross $c$, such as $\ttarc{I_0I_1}$, is fixed by $t_c$.}
\label{fig:rule}
\end{figure}

\begin{example}\label{ex:twist}
Let $n=7$, $c=c_{3,6}$ and $\gamma=c_{1,5}$, so that $\kappa=2$, $\nu=6$ and $w(\gamma)=\ttarc{I_0I_5}\ssarc{I_5I_0}$. Both arcs of $\gamma$ cross $c$, since their endpoint in $I_0$ is outer and their endpoint in $I_5$ is inner with respect to $c$. Since $T^1_U=(2,6)$ and $T^1_D=(6,2)$, Proposition~\ref{prop:twist} gives
\[
\ttarc{I_0I_5}\longmapsto\ttarc{I_0I_2}\ssarc{I_2I_6}\ttarc{I_6I_5},\qquad \ssarc{I_0I_5}\longmapsto\ssarc{I_0I_6}\ttarc{I_6I_2}\ssarc{I_2I_5}.
\]
Reading the second chain from $I_5$ to $I_0$ and concatenating the two chains, we obtain
\[
w(t_c(\gamma))=\ttarc{I_0I_2}\ssarc{I_2I_6}\ttarc{I_6I_5}\ssarc{I_5I_2}\ttarc{I_2I_6}\ssarc{I_6I_0},
\]
which is reduced. The upper arcs of this curve are $\ttarc{I_0I_2}$, $\ttarc{I_6I_5}$ and $\ttarc{I_2I_6}$, and its lower arcs are $\ssarc{I_2I_6}$, $\ssarc{I_5I_2}$ and $\ssarc{I_6I_0}$. It contains the upper arc $\ttarc{I_2I_6}$ and the lower arc $\ssarc{I_2I_6}$ of $c$, so $t_c(\gamma)$ lies in the witness set $Q_c=Q_{3,6}$ of $c$ (Definition~\ref{def:witness}); see Figure~\ref{fig:example}(b).
\end{example}

\subsection{Crossings that survive}
In the ping-pong arguments of Section~\ref{sec:A} we apply twists one after another, and after the first twist the curve is no longer relaxed. Proposition~\ref{prop:twist} gives a chain of the image, but this chain is usually not reduced, and Move~1 might destroy the arcs we are interested in. The next lemma shows that all crossings outside the endpoint intervals of the twist curves survive. Lemma~\ref{lem:arcs} then shows that arcs whose endpoints are both inner or both outer with respect to $c$ survive as well. In this subsection and in Section~\ref{sec:witness} we apply this to a single twist $t_c^m$; in Section~\ref{sec:B} we apply it to products $t_a^mt_b^{m'}$, and for this reason we state it for products of twists. Let $\mathcal E$ be a set of indices of intervals, that is, $\mathcal E\subset\{0,\dots,n\}$. We say that a mapping class $g$ is a \emph{twist product over $\mathcal E$} if it is a product of powers of twists about relaxed curves whose endpoint intervals all belong to $\{I_e: e\in\mathcal E\}$. For instance $t_c^m$ is a twist product over $\{\kappa,\nu\}$, the indices of the endpoint intervals of $c$. Let $\gamma\in\mathcal S(D_n)$ be a curve, and suppose that $\gamma$ crosses $H$ in some interval $I_s$ with $s\notin\mathcal E$. Let $z_1,\dots,z_N$ be the crossings of $\gamma$ with $H$ in such intervals, in the order in which they appear along $\gamma$. These points cut $\gamma$ into $N$ paths: for $1\le h\le N$ let $\omega_h$ be the path along $\gamma$ from $z_h$ to $z_{h+1}$, where $z_{N+1}=z_1$. All crossings of $\omega_h$ with $H$ other than its endpoints lie in the intervals $I_e$, $e\in\mathcal E$. Let $g$ be a twist product over $\mathcal E$. As in Section~\ref{sec:twistrule}, we choose the annuli of the twists in $g$ so thin that they meet $H$ only in the intervals $I_e$, $e\in\mathcal E$, and do not contain the points $z_h$, and we represent $g$ by a homeomorphism which is the identity outside these annuli. So $g$ fixes each $z_h$, and $g(\omega_h)$ is a path from $z_h$ to $z_{h+1}$ (Figure~\ref{fig:survive}). We write $\Gamma_h$ for the reduced chain of the path $g(\omega_h)$, that is, the sequence of crossings of $g(\omega_h)$ with $H$ between its endpoints after applying Move~1 as long as possible. So $g(\gamma)$ is the union of the paths $g(\omega_1),\dots,g(\omega_N)$, joined at the points $z_1,\dots,z_N$, and the next lemma says that reducing each of these paths separately gives the reduced chain of $g(\gamma)$.

\begin{figure}[ht]
\centering
\labellist
\hair 2pt
\pinlabel {{\scriptsize \textcolor{black!60}{$c$}}} at 91 116
\pinlabel {{\ilab \textcolor{red!80!black}{$z_{1}$}}} at 23 97
\pinlabel {{\ilab \textcolor{red!80!black}{$z_{2}$}}} at 92 97
\pinlabel {{\ilab \textcolor{red!80!black}{$z_{3}$}}} at 159 97
\pinlabel {(a)} at 91 5
\pinlabel {{\scriptsize \textcolor{black!60}{$c$}}} at 359 116
\pinlabel {{\ilab \textcolor{red!80!black}{$z_{1}$}}} at 290 97
\pinlabel {{\ilab \textcolor{red!80!black}{$z_{2}$}}} at 360 97
\pinlabel {{\ilab \textcolor{red!80!black}{$z_{3}$}}} at 427 97
\pinlabel {(b)} at 359 5
\pinlabel {{\scriptsize \textcolor{blue!70!black}{$\omega_1$}}} at 36 103
\pinlabel {{\scriptsize \textcolor{orange!85!black}{$\omega_2$}}} at 149 112
\pinlabel {{\scriptsize \textcolor{green!45!black}{$\omega_3$}}} at 91 20
\pinlabel {{\scriptsize \textcolor{blue!70!black}{$\Gamma_1$}}} at 301 100
\pinlabel {{\scriptsize \textcolor{orange!85!black}{$\Gamma_2$}}} at 395 142
\pinlabel {{\scriptsize \textcolor{green!45!black}{$\Gamma_3$}}} at 359 20
\pinlabel {{\ilab $I_{0}$}} at 17 84
\pinlabel {{\ilab $I_{1}$}} at 51 84
\pinlabel {{\ilab $I_{2}$}} at 77 84
\pinlabel {{\ilab $I_{3}$}} at 108 84
\pinlabel {{\ilab $I_{4}$}} at 133 84
\pinlabel {{\ilab $I_{5}$}} at 161 84
\pinlabel {{\ilab $I_{0}$}} at 285 84
\pinlabel {{\ilab $I_{1}$}} at 317 84
\pinlabel {{\ilab $I_{2}$}} at 345 84
\pinlabel {{\ilab $I_{3}$}} at 372 84
\pinlabel {{\ilab $I_{4}$}} at 400 84
\pinlabel {{\ilab $I_{5}$}} at 428 84
\endlabellist
\includegraphics{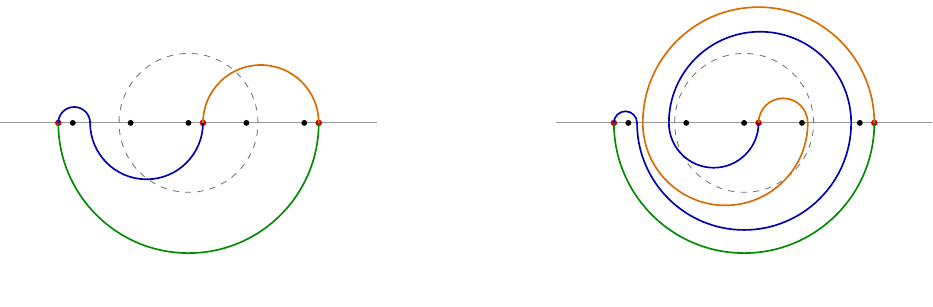}
\caption{The paths $\omega_h$ and their images for $g=t_c$ with $c=c_{2,4}$ in $D_5$, so that $\mathcal E=\{1,4\}$. (a) The curve $\gamma$ with chain $\ttarc{I_0I_1}\ssarc{I_1I_3}\ttarc{I_3I_5}\ssarc{I_5I_0}$. Its crossings $z_1\in I_0$, $z_2\in I_3$ and $z_3\in I_5$ lie outside the intervals $I_1$ and $I_4$, and they cut $\gamma$ into the paths $\omega_1$, $\omega_2$ and $\omega_3$. (b) The curve $t_c(\gamma)$. The crossings $z_1$, $z_2$, $z_3$ are unchanged. Between them, $\Gamma_1$ crosses $H$ in $I_1,I_4,I_1$, $\Gamma_2$ crosses $H$ in $I_4,I_1$, and $\Gamma_3=\omega_3$ does not cross $H$.}
\label{fig:survive}
\end{figure}

\begin{lemma}\label{lem:survive}
Let $g$ be a twist product over $\mathcal E$, and let $\gamma$, $z_h$ and $\omega_h$ be as above. Then the reduced chain of $g(\gamma)$ has the crossings $z_1,\dots,z_N$ in the same intervals, in the same cyclic order and with the same directions, and between $z_h$ and $z_{h+1}$ it is $\Gamma_h$, whose crossings all lie in intervals $I_e$ with $e\in\mathcal E$.
\end{lemma}

\begin{proof}
Since $\omega_h$ crosses $H$ only in the intervals $I_e$ and the twists in $g$ change curves only inside the annuli, $g(\omega_h)$ also crosses $H$ only in the intervals $I_e$. Move~1 applied to one of the paths $g(\omega_h)$ is also Move~1 applied to $g(\gamma)$. Hence the chain of $g(\gamma)$ reduces to the chain which crosses $H$ at $z_1$, then at the crossings of $\Gamma_1$, then at $z_2$, and so on, ending with the crossings of $\Gamma_N$. Move~1 can only be applied to two consecutive crossings in the same interval. The crossings of $\Gamma_h$ lie in the intervals $I_e$ and $z_h$ does not, so the only possibility is that some $\Gamma_h$ is empty and $z_h$, $z_{h+1}$ lie in the same interval $I_s$. We show that this does not happen. If $N=1$, then $\Gamma_1$ is not empty. Indeed, Move~1 removes crossings in pairs and $g(\gamma)$ crosses $H$ an even number of times, so $\Gamma_1$ has an odd number of crossings. Let $N\ge2$ and suppose that $\Gamma_h$ is empty. Since Move~1 removes crossings in pairs, $g(\omega_h)$ crosses $H$ an even number of times between its endpoints, so it leaves $z_h$ and arrives at $z_{h+1}$ in the same half disk. The same holds for $\omega_h$, which agrees with $g(\omega_h)$ near these points. Let $\sigma$ be the segment of $I_s$ between $z_h$ and $z_{h+1}$. Choose a path $\tau$ close to $\sigma$ in this half disk, agreeing with $\omega_h$ near its endpoints. We may choose $\tau$ sufficiently close to $\sigma$ that $g$ fixes it. Since $\Gamma_h$ is empty, $g(\omega_h)$ and $\tau$ have the same reduced chain, so they are homotopic relative to neighbourhoods of their endpoints (Section~\ref{sec:chains}). Applying $g^{-1}$, the path $\omega_h$ is homotopic to $\tau$ in the same sense, so the reduced chain of $\omega_h$ is empty. The chain of $\omega_h$ is reduced, since $w(\gamma)$ is, so $\omega_h$ is a single arc $\ttarc{I_sI_s}$ or $\ssarc{I_sI_s}$, which contradicts the fact that $w(\gamma)$ is reduced. Hence the chain through $z_1,\Gamma_1,\dots,z_N,\Gamma_N$ is reduced, and by Lemma~\ref{lem:reduced} it is the reduced chain of $g(\gamma)$.
\end{proof}
The next lemma says that $t_c^m$ does not destroy the arcs of a curve whose endpoints are both inner or both outer with respect to $c$ (Figure~\ref{fig:arcs}). We use it in Corollary~\ref{cor:nonalt}: if $c_{k,\ell}$ does not alternate with $c$ and has no endpoint interval in common with $c$, then the arcs $\ttarc{I_{k-1}I_\ell}$ and $\ssarc{I_{k-1}I_\ell}$ have both endpoints inner or both endpoints outer with respect to $c$, so $t_c^m$ maps $Q_{k,\ell}$ into itself.

\begin{lemma}\label{lem:arcs}
Let $c$ be a relaxed curve, let $\gamma\in\mathcal S(D_n)$ be a curve and let $m\in\Z$. Then every arc of $\gamma$ whose endpoints are both inner or both outer with respect to $c$ is an arc of $t_c^m(\gamma)$, and distinct such arcs give distinct arcs of $t_c^m(\gamma)$.
\end{lemma}

\begin{proof}
Such an arc $\delta$ does not cross $c$, so $t_c^m(\delta)=\delta$ by Proposition~\ref{prop:twist}. Its endpoints lie outside $I_\kappa\cup I_\nu$, so $\delta$ is one of the paths $\omega_h$ of Lemma~\ref{lem:survive} for $\mathcal E=\{\kappa,\nu\}$, with $\Gamma_h=\delta$, and the lemma follows from Lemma~\ref{lem:survive}.
\end{proof}

\begin{figure}[ht]
\centering
\labellist
\hair 2pt
\pinlabel {{\ilab $I_{0}$}} at 16 81
\pinlabel {{\ilab $I_{1}$}} at 48 81
\pinlabel {{\ilab $I_{2}$}} at 75 81
\pinlabel {{\ilab $I_{3}$}} at 102 81
\pinlabel {{\ilab $I_{4}$}} at 129 81
\pinlabel {{\ilab $I_{5}$}} at 156 81
\pinlabel {{\scriptsize \textcolor{green!45!black}{$\delta$}}} at 81 67
\pinlabel {{\scriptsize \textcolor{green!45!black}{$\delta'$}}} at 88 20
\pinlabel {(a)} at 88 5
\pinlabel {{\ilab $I_{0}$}} at 276 81
\pinlabel {{\ilab $I_{1}$}} at 307 81
\pinlabel {{\ilab $I_{2}$}} at 334 81
\pinlabel {{\ilab $I_{3}$}} at 361 81
\pinlabel {{\ilab $I_{4}$}} at 388 81
\pinlabel {{\ilab $I_{5}$}} at 415 81
\pinlabel {{\scriptsize \textcolor{green!45!black}{$\delta$}}} at 341 67
\pinlabel {{\scriptsize \textcolor{green!45!black}{$\delta'$}}} at 348 20
\pinlabel {(b)} at 348 5
\pinlabel {{\scriptsize \textcolor{black!60}{$c$}}} at 88 124
\endlabellist
\includegraphics{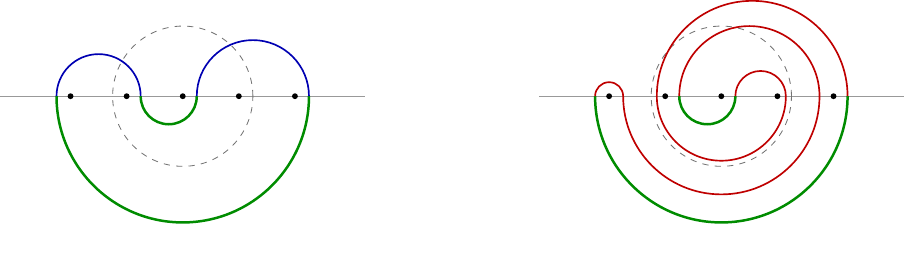}
\caption{Lemma~\ref{lem:arcs} for $c=c_{2,4}$ in $D_5$. (a) The curve $\gamma$ with chain $\ttarc{I_0I_2}\ssarc{I_2I_3}\ttarc{I_3I_5}\ssarc{I_5I_0}$. The arc $\delta=\ssarc{I_2I_3}$ has both endpoints inner and $\delta'=\ssarc{I_5I_0}$ has both endpoints outer with respect to $c$, so they do not cross $c$ (green). The two upper arcs cross $c$ (blue). (b) The curve $t_c(\gamma)$. The arcs $\delta$ and $\delta'$ are unchanged, while the upper arcs of $\gamma$ go once around $c$ (red).}
\label{fig:arcs}
\end{figure}

\subsection{Witness arcs}\label{sec:witness}
Recall from Definition~\ref{def:witness} that if $c_{k,\ell}$ is a relaxed curve which alternates with $c$, then an arc $\ttarc{I_{k-1}I_{\ell}}$ or $\ssarc{I_{k-1}I_{\ell}}$ is called a \emph{witness arc} for $c$; we also call it simply a \emph{witness}. One of the intervals $I_{k-1}$, $I_\ell$ is inner with respect to $c$ and the other is outer, so a witness arc crosses $c$. By Lemma~\ref{lem:witness} below, if a curve $\gamma$ has both an upper and a lower witness arc for $c$, then $t_c^m(\gamma)\in Q_c$ for every $m\ne0$. A single witness arc is enough when $|m|\ge2$. The image of a witness under a twist about $c$ goes around $c$. A single twist about $c$ turns an upper witness into a chain containing a lower arc $\ssarc{I_{\kappa}I_{\nu}}$ of $c$ and a lower witness into a chain containing an upper arc $\ttarc{I_{\kappa}I_{\nu}}$ of $c$, as in Example~\ref{ex:twist}, and higher powers produce more such arcs.

\begin{lemma}[Witness lemma]\label{lem:witness}
Let $c=c_{i,j}$ and $c_{k,\ell}$ be relaxed curves which alternate, let $\gamma\in\mathcal S(D_n)$ be a curve, let $m\neq0$ and write $\gamma'=t_c^m(\gamma)$. Then
\begin{align*}
\mu_{\gamma'}(\ttarc{I_{\kappa}I_{\nu}})&\ge (|m|-1)\,\mu_\gamma(\ttarc{I_{k-1}I_{\ell}})+|m|\,\mu_\gamma(\ssarc{I_{k-1}I_{\ell}}),\\
\mu_{\gamma'}(\ssarc{I_{\kappa}I_{\nu}})&\ge |m|\,\mu_\gamma(\ttarc{I_{k-1}I_{\ell}})+(|m|-1)\,\mu_\gamma(\ssarc{I_{k-1}I_{\ell}}).
\end{align*}
The same bounds hold after summing the right hand sides over any set of distinct relaxed curves alternating with $c$.
\end{lemma}

\begin{proof}
A witness has one inner and one outer endpoint with respect to $c$, so it is one of the paths $\omega_h$ of Lemma~\ref{lem:survive} for $\mathcal E=\{\kappa,\nu\}$, consisting of a single arc. By Lemma~\ref{lem:survive} it is replaced in the reduced chain of $t_c^m(\gamma)$ by the reduced chain of its image. Let $m\ge1$, and let $\ttarc{I_qI_p}$ be an upper witness with $I_q$ outer and $I_p$ inner. By Proposition~\ref{prop:twist} it becomes
\[
\ttarc{I_qI_{\kappa}}\bigl(\ssarc{I_{\kappa}I_{\nu}}\ttarc{I_{\nu}I_{\kappa}}\bigr)^{m-1}\ssarc{I_{\kappa}I_{\nu}}\ttarc{I_{\nu}I_p}.
\]
Its crossings between $I_q$ and $I_p$ alternate between $I_\kappa$ and $I_\nu$, so no two consecutive crossings lie in the same interval, and this chain is already reduced. The first and the last arc do not join $I_\kappa$ and $I_\nu$, since $I_q$ and $I_p$ are not boundary. The other arcs are $m$ lower arcs $\ssarc{I_{\kappa}I_{\nu}}$ and $m-1$ upper arcs $\ttarc{I_{\kappa}I_{\nu}}$. For instance, for $m=2$ the chain is $\ttarc{I_qI_\kappa}\ssarc{I_\kappa I_\nu}\ttarc{I_\nu I_\kappa}\ssarc{I_\kappa I_\nu}\ttarc{I_\nu I_p}$, with two lower arcs and one upper arc. Similarly, a lower witness gives $m$ upper arcs and $m-1$ lower arcs. For $m<0$ Proposition~\ref{prop:twist} only exchanges $\kappa$ and $\nu$, so the same counts hold with $|m|$ in place of $m$. Different witnesses are different paths $\omega_h$, so nothing is counted twice. Adding the contributions of all upper and all lower witnesses gives the two inequalities.
\end{proof}

Lemma~\ref{lem:witness} only gives lower bounds. Other arcs of $\gamma$ may also give arcs joining $I_\kappa$ and $I_\nu$ in $t_c^m(\gamma)$: for instance an arc of $\gamma$ which crosses $c$ and has a boundary endpoint also goes around $c$, and an arc $\ttarc{I_\kappa I_\nu}$ or $\ssarc{I_\kappa I_\nu}$ of $\gamma$ may remain. We do not count these, since the lower bounds are sufficient for the ping-pong arguments.

\begin{remark}\label{rem:dynnikovwitness}
By Remark~\ref{prop:dynnikov}, the measures in Lemma~\ref{lem:witness} are explicit functions of Dynnikov coordinates. If $(a;b)$ are the Dynnikov coordinates of $\gamma$, then those of $t_{i,j}^m(\gamma)$ are obtained from $(a;b)$ by writing $t_{i,j}=\Delta^2_{i,j}$ as a word in the generators $\sigma_s$ and applying the update rules of \cite{Dynnikov,HY09,HY18}. So both sides of the inequalities in Lemma~\ref{lem:witness} can be computed directly from Dynnikov coordinates. In the proofs, however, we work with chains and not with Dynnikov coordinates, since the action of a twist and the cancellations of Move~1 are easy to see on the chain, while the formulae for $t_{i,j}$ in Dynnikov coordinates are piecewise linear and hard to follow along a product of twists.
\end{remark}
Lemmas~\ref{lem:witness} and~\ref{lem:arcs} have the following consequences for witness sets.

\begin{corollary}\label{cor:alt}
If $c_{i,j}$ and $c_{k,\ell}$ alternate, then $t_{i,j}^m(Q_{k,\ell})\subset Q_{i,j}$ for every $m\neq0$.
\end{corollary}

\begin{proof}
For $\gamma\in Q_{k,\ell}$ both $\mu_\gamma(\ttarc{I_{k-1}I_{\ell}})$ and $\mu_\gamma(\ssarc{I_{k-1}I_{\ell}})$ are positive, so Lemma~\ref{lem:witness} gives
\begin{align*}
\mu_{\gamma'}(\ttarc{I_{i-1}I_j})&\ge \mu_\gamma(\ssarc{I_{k-1}I_{\ell}})>0,\\
\mu_{\gamma'}(\ssarc{I_{i-1}I_j})&\ge\mu_\gamma(\ttarc{I_{k-1}I_{\ell}})>0
\end{align*}
for $\gamma'=t^m_{i,j}(\gamma)$.
\end{proof}

\begin{corollary}\label{cor:nonalt}
If $c_{i,j}$ and $c_{k,\ell}$ do not alternate and have four distinct endpoint intervals, then $t_{i,j}^m(Q_{k,\ell})\subset Q_{k,\ell}$ for every $m\in\Z$.
\end{corollary}

\begin{proof}
The intervals $I_{k-1}$ and $I_\ell$ are not boundary with respect to $c_{i,j}$, and since the two curves are nested or side by side they are both inner or both outer. Let $\gamma\in Q_{k,\ell}$. Then $\gamma$ has an arc $\ttarc{I_{k-1}I_{\ell}}$ and an arc $\ssarc{I_{k-1}I_{\ell}}$, and by Lemma~\ref{lem:arcs} both are arcs of $t^m_{i,j}(\gamma)$. Hence $t^m_{i,j}(\gamma)\in Q_{k,\ell}$.
\end{proof}

\subsection{Pairs with a shared endpoint interval}\label{sec:B}
In this subsection $a$ and $b$ are relaxed curves which do not alternate and have exactly one common endpoint interval. Then one curve surrounds the other or they lie side by side, and there are two configurations.
\begin{itemize}
\item[(N)] \emph{Nested}: $a$ surrounds $b$. Let $I_x$ be the common endpoint interval and $I_{u}$, $I_{v}$ the other endpoint intervals of $a$ and $b$; either $u<v<x$, so that the common endpoint interval is on the right (case N$_R$), or $x<v<u$, so that it is on the left (case N$_L$). For example, $c_{1,4}$ and $c_{2,4}$ share $I_4$ (case N$_R$, Figure~\ref{fig:paired}(a)), and $c_{1,4}$ and $c_{1,3}$ share $I_0$ (case N$_L$).
\item[(S)] \emph{Side by side}: the curves lie side by side, and $a$ is the left one. The endpoint intervals are $I_u,I_x$ for $a$ and $I_x,I_v$ for $b$, with $u<x<v$. For example, $c_{1,2}$ and $c_{3,4}$ share $I_2$.
\end{itemize}
Since $a$ and $b$ are disjoint, $t_a$ and $t_b$ commute, so Lemma~\ref{lem:paired} below is symmetric in $a$ and $b$ and we may name the curves as above. In both cases $I_x$ is the common endpoint interval, $I_u$ is the other endpoint interval of $a$ and $I_v$ is the other endpoint interval of $b$. Let $\mathcal E=\{u,v,x\}$ be the set of the three indices.

We first explain the idea. Let $c$ be a curve which alternates with both $a$ and $b$, and let $\gamma\in Q_c$. The upper arc of $c$ in $\gamma$ crosses both $a$ and $b$. Its image under $t_a^mt_b^{m'}$ goes $|m|$ times around $a$ and $|m'|$ times around $b$. So it has $2|m|$ new crossings in the endpoint intervals of $a$ and $2|m'|$ new crossings in those of $b$. The only interval which $a$ and $b$ have in common is the shared one. Hence these crossings can cancel only where the two groups meet, and at most one pair cancels. If $|m|\ge2$ or $|m'|\ge2$, what remains still contains both an upper and a lower arc of $a$, or of $b$. The only critical case is $|m|=|m'|=1$. For $(m,m')=\pm(1,-1)$, the upper and lower witness images need not provide both arc types of either curve. These are the exceptional cases $g=t_at_b^{-1}$ and its inverse.

Figure~\ref{fig:paired} shows an example. In Figure~\ref{fig:paired}(a) the curves $a=c_{1,4}$ and $b=c_{2,4}$ share the interval $I_4$, and $c=c_{3,5}$ alternates with both. The upper arc $\delta=\ttarc{I_5I_2}$ of $c$ joins $I_5$, which is outer with respect to $a$ and $b$, to $I_2$, which is inner with respect to both. Going once around $a$ adds crossings in $I_0$ and $I_4$, and going once around $b$ adds crossings in $I_1$ and $I_4$. Under $t_at_b$ nothing cancels, and $\delta$ is replaced by $\ttarc{I_5I_0}\ssarc{I_0I_4}\ttarc{I_4I_1}\ssarc{I_1I_4}\ttarc{I_4I_2}$, which contains both arcs $\ttarc{I_1I_4}$ and $\ssarc{I_1I_4}$ of $b$ (Figure~\ref{fig:paired}(b)). Under $t_at_b^{-1}$ the two crossings in $I_4$ are consecutive and cancel by Move~1 (Figure~\ref{fig:paired}(c)), and $\delta$ is replaced by $\ttarc{I_5I_0}\ssarc{I_0I_1}\ttarc{I_1I_2}$, which contains no arc of $a$ or $b$.

\begin{lemma}\label{lem:thirdcurve}
Let $c$ be a relaxed curve which alternates with both $a$ and $b$. Then we can write $I_{y_1}$, $I_{y_2}$ for its endpoint intervals so that the following holds.
\begin{itemize}
\item[(N)] $I_{y_2}$ is inner and $I_{y_1}$ is outer with respect to both $a$ and $b$.
\item[(S)] $I_{y_1}$ is inner with respect to $a$ and outer with respect to $b$, and $I_{y_2}$ is outer with respect to $a$ and inner with respect to $b$; that is, $u<y_1<x<y_2<v$.
\end{itemize}
\end{lemma}

\begin{proof}
By definition, $c_{k,\ell}$ alternates with $c_{i,j}$ if and only if one of $I_{k-1}$, $I_\ell$ is inner and the other is outer with respect to $c_{i,j}$. So the endpoint intervals of $c$ are not boundary with respect to $a$ or $b$. In case (N) the curve $a$ surrounds $b$, so every interval which is inner with respect to $b$ is also inner with respect to $a$. Since $c$ alternates with $b$, exactly one endpoint interval of $c$ is inner with respect to $b$; we call it $I_{y_2}$. Then $I_{y_2}$ is also inner with respect to $a$. Since $c$ alternates with $a$, the other endpoint interval $I_{y_1}$ is outer with respect to $a$. It lies outside $a$, so it is also outer with respect to $b$.

In case (S) the inner intervals with respect to $a$ are $I_{u+1},\dots,I_{x-1}$, and those with respect to $b$ are $I_{x+1},\dots,I_{v-1}$, so no interval is inner with respect to both. Exactly one endpoint interval of $c$ is inner with respect to $a$, and we call it $I_{y_1}$. Exactly one is inner with respect to $b$, and we call it $I_{y_2}$. Thus $u<y_1<x<y_2<v$. The interval $I_{y_1}$ is not inner with respect to $b$ and not an endpoint interval of $b$, so it is outer with respect to $b$. Similarly $I_{y_2}$ is outer with respect to $a$.
\end{proof}

Let $c$ be as in Lemma~\ref{lem:thirdcurve}, let $\gamma\in\mathcal S(D_n)$ be a curve with an upper arc $\delta=\ttarc{I_{y_1}I_{y_2}}$ of $c$, and let $g=t_a^mt_b^{m'}$ with $m,m'\in\Z$. We write $W$ for the sequence of indices of the intervals in which the reduced chain of $g(\delta)$, read from $I_{y_1}$ to $I_{y_2}$, crosses $H$ between its endpoints. The next lemma computes $W$; Figure~\ref{fig:paired} shows two examples.

\begin{figure}[ht]
\centering
\labellist
\hair 2pt
\pinlabel {{\scriptsize $a$}} at 133 299
\pinlabel {{\scriptsize $b$}} at 169 297
\pinlabel {{\scriptsize $c$}} at 265 297
\pinlabel {{\scriptsize \textcolor{red!75!black}{$\delta$}}} at 245 290
\pinlabel {(a)} at 216 190
\pinlabel {{\scriptsize $a$}} at 19 120
\pinlabel {{\scriptsize $b$}} at 59 117
\pinlabel {{\scriptsize $a$}} at 246 120
\pinlabel {{\scriptsize $b$}} at 283 117
\pinlabel {(b)} at 103 15
\pinlabel {(c)} at 330 15
\pinlabel {{\ilab $I_{0}$}} at 122 273
\pinlabel {{\ilab $I_{1}$}} at 156 273
\pinlabel {{\ilab $I_{2}$}} at 185 273
\pinlabel {{\ilab $I_{3}$}} at 216 273
\pinlabel {{\ilab $I_{4}$}} at 248 273
\pinlabel {{\ilab $I_{5}$}} at 276 273
\pinlabel {{\ilab $I_{6}$}} at 302 273
\pinlabel {{\ilab $I_{0}$}} at 10 93
\pinlabel {{\ilab $I_{1}$}} at 46 93
\pinlabel {{\ilab $I_{2}$}} at 71 93
\pinlabel {{\ilab $I_{3}$}} at 103 93
\pinlabel {{\ilab $I_{4}$}} at 134 90
\pinlabel {{\ilab $I_{5}$}} at 162 93
\pinlabel {{\ilab $I_{6}$}} at 188 93
\pinlabel {{\ilab $I_{0}$}} at 235 93
\pinlabel {{\ilab $I_{1}$}} at 274 93
\pinlabel {{\ilab $I_{2}$}} at 299 93
\pinlabel {{\ilab $I_{3}$}} at 330 93
\pinlabel {{\ilab $I_{4}$}} at 360 90
\pinlabel {{\ilab $I_{5}$}} at 390 93
\pinlabel {{\ilab $I_{6}$}} at 416 93
\endlabellist
\includegraphics{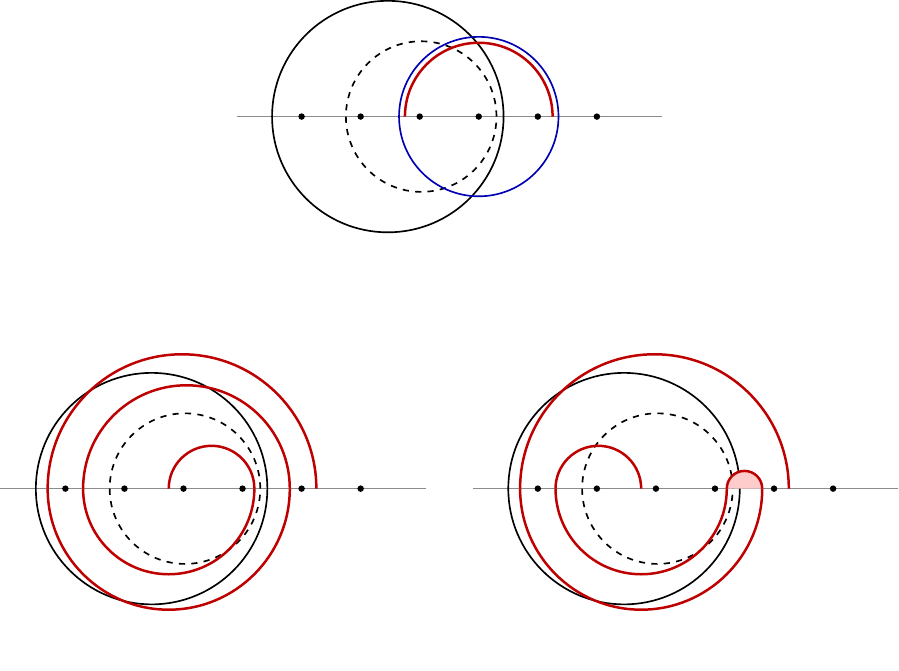}
\caption{(a) The curves $a=c_{1,4}$, $b=c_{2,4}$ (dashed) and $c=c_{3,5}$ (blue) in $D_6$, and the upper arc $\delta$ of $c$ (red). (b) The image of $\delta$ under $t_at_b$; here $W=(0,4,1,4)$. (c) The image of $\delta$ under $t_at_b^{-1}$, before Move~1; it crosses $H$ in $I_0,I_4,I_4,I_1$. The upper arc $\ttarc{I_4I_4}$ between the two consecutive crossings in $I_4$ bounds a bigon with $H$ (shaded). Move~1 removes the two crossings, and $W=(0,1)$.}
\label{fig:paired}
\end{figure}

\begin{lemma}\label{lem:W}
Let $c$, $\gamma$, $\delta$, $g$ and $W$ be as above. Then in the reduced chain of $g(\gamma)$ the arc $\delta$ is replaced by the reduced chain of $g(\delta)$, and $W$ is obtained by cancelling consecutive equal entries in
\[
T^m_{a,U}\,T^{m'}_{b,U}\ \text{ in case (N)},\qquad
T^{-m}_{a,U}\,T^{m'}_{b,U}\ \text{ in case (S)}.
\]
\end{lemma}

\begin{proof}
Since $y_1,y_2\notin\mathcal E$, the arc $\delta$ is one of the paths $\omega_h$ of Lemma~\ref{lem:survive}, so the first statement is Lemma~\ref{lem:survive}. Choose disjoint annuli supporting $t_a$ and $t_b$; by Lemma~\ref{lem:thirdcurve} each is crossed once by $\delta$. Following $\delta$ from $I_{y_1}$, in case (N) we enter $a$ and then $b$ (Figure~\ref{fig:paired}(a)), and the crossings inside the annuli are $T^m_{a,U}$ followed by $T^{m'}_{b,U}$; the second sequence is the $U$ version since $T^m_{a,U}$ has even length. In case (S) we leave $a$ and then enter $b$; leaving $a$ reverses the sequence, giving $T^{-m}_{a,U}$. Applying Move~1 gives $W$.
\end{proof}

Figure~\ref{fig:paired} shows Lemma~\ref{lem:W} for $a=c_{1,4}$, $b=c_{2,4}$ and $c=c_{3,5}$. In Figure~\ref{fig:paired}(b), $g=t_at_b$ and the sequence is $T^1_{a,U}T^1_{b,U}=(0,4,1,4)$; no two consecutive entries are equal, so $W=(0,4,1,4)$. In Figure~\ref{fig:paired}(c), $g=t_at_b^{-1}$ and the sequence is $T^1_{a,U}T^{-1}_{b,U}=(0,4,4,1)$; the two consecutive entries $4$ are the endpoints of the upper arc $\ttarc{I_4I_4}$, which bounds a bigon with $H$, and Move~1 removes them, so $W=(0,1)$.

The next lemma is the version of Corollary~\ref{cor:alt} for a pair of twists. We cannot apply Corollary~\ref{cor:alt} twice, since the crossings created by $t_a$ and by $t_b$ may cancel in the common interval. The image of a witness arc of $c$ goes $|m|$ times around $a$ and $|m'|$ times around $b$, and we show that enough of the crossings created in this way survive Move~1 to give both an upper and a lower arc of $a$, or of $b$.

\begin{lemma}[Paired witness lemma]\label{lem:paired}
Let $a$, $b$ and $c$ be as above and let $(m,m')\in\Z^2$ with $(m,m')\ne(0,0)$ and $(m,m')\ne\pm(1,-1)$. Then $t_a^mt_b^{m'}(Q_c)\subset Q_a\cup Q_b$.
\end{lemma}

\begin{proof}
Write $g=t_a^mt_b^{m'}$ and let $\gamma\in Q_c$. If $m'=0$, then $m\ne0$ and $g(\gamma)=t_a^m(\gamma)\in Q_a$ by Corollary~\ref{cor:alt}, since $c$ alternates with $a$; similarly $g(\gamma)\in Q_b$ if $m=0$. So let $m,m'\ne0$. We apply Lemma~\ref{lem:W} to the upper arc $\delta$ of $c$ in $\gamma$; every arc between two consecutive crossings of $W$ is an arc of $g(\gamma)$. Write $T_a$ for the first sequence in Lemma~\ref{lem:W}, that is, $T_a=T^m_{a,U}$ in case (N) and $T_a=T^{-m}_{a,U}$ in case (S), and $T_b=T^{m'}_{b,U}$. The $2|m|$ entries of $T_a$ alternate between $u$ and $x$, and the $2|m'|$ entries of $T_b$ alternate between $x$ and $v$. So the only two consecutive equal entries of $T_aT_b$ can be the last entry of $T_a$ and the first entry of $T_b$, when both are $x$, and after cancelling them the new neighbours are $u$ and $v$. Hence $W$ is obtained from $T_aT_b$ by cancelling at most this pair. If $|m|\ge2$, then at least three consecutive entries of $T_a$ remain in $W$, and the two arcs between them are an upper arc and a lower arc of $a$, so $g(\gamma)\in Q_a$. Similarly $g(\gamma)\in Q_b$ if $|m'|\ge2$. It remains to consider $(m,m')=(1,1)$ and $(-1,-1)$. Since $T^1_U=(\kappa,\nu)$ and $T^{-1}_U=(\nu,\kappa)$ with $\kappa<\nu$, the last entry of $T_a$ is $x$ for exactly one sign of $m$, and the first entry of $T_b$ is $x$ for exactly the opposite sign of $m'$: in case N$_R$ for $m=1$ and $m'=-1$, and in cases N$_L$ and S for $m=-1$ and $m'=1$. So for $(1,1)$ and $(-1,-1)$ nothing cancels and $W=T_aT_b$, which is $(u,x,v,x)$ if $T_a$ ends with $x$ and $(x,u,x,v)$ otherwise. The arcs between the crossings of $W$ are
\[
\ssarc{I_{u}I_{x}}\ttarc{I_{x}I_{v}}\ssarc{I_{v}I_{x}}\qquad\text{and}\qquad\ssarc{I_{x}I_{u}}\ttarc{I_{u}I_{x}}\ssarc{I_{x}I_{v}},
\]
respectively. The first chain contains the upper arc $\ttarc{I_xI_v}$ and the lower arc $\ssarc{I_vI_x}$ of $b$, and the second one the upper arc $\ttarc{I_uI_x}$ and the lower arc $\ssarc{I_xI_u}$ of $a$. Hence $g(\gamma)\in Q_a\cup Q_b$; see Figure~\ref{fig:paired}(b) for the case N$_R$ with $(m,m')=(1,1)$.
\end{proof} For $(m,m')=\pm(1,-1)$ the crossings in the common interval $I_x$ cancel in exactly one of the upper and lower witness images. Its two remaining crossings lie in $I_u$ and $I_v$, so the arc between them is neither an arc of $a$ nor an arc of $b$. In the other witness image no crossings cancel, and the arcs of $a$ and $b$ lie in the same half disk. For example, in case N$_R$ the sequences for the upper witness are $(u,x,x,v)$ for $(m,m')=(1,-1)$ and $(x,u,v,x)$ for $(m,m')=(-1,1)$; only the first reduces to $(u,v)$. The calculation for the lower witness is obtained by interchanging $U$ and $D$. Taking $\gamma=c$ shows that the inclusion fails in each exceptional case: there are no other arcs to supply the missing witness type. Example~\ref{ex:D6} gives the reduced chains explicitly.

\begin{remark}\label{rem:lantern}
The exceptional elements of Lemma~\ref{lem:paired} are related to the lantern relation \cite[Section~5.1]{FM}. For example, let $a=c_{1,3}$ and $b=c_{1,2}$ in $D_n$, $n\ge4$, which share the endpoint interval $I_0$. The lantern relation in the disk bounded by $c_{1,3}$ gives
\[
t_{1,3}=t_{1,2}\,t_{2,3}\,t_x,
\]
where $x$ is a curve which surrounds the first and the third punctures and crosses $H$ four times. Since $t_{1,3}$ and $t_{1,2}$ commute, we get $t_at_b^{-1}=t_{2,3}t_x$. So the exceptional element $t_at_b^{-1}$ is a product of twists about two other curves, one of which is not relaxed. Similarly $t_{2,4}=t_{3,4}\,t_{2,3}\,t_y$, where $y$ surrounds the second and the fourth punctures, and the curves $x$ and $y$ are disjoint. Hence $t_{1,3}^{-1}t_{1,2}t_{2,3}=t_x^{-1}$ and $t_{2,4}^{-1}t_{3,4}t_{2,3}=t_y^{-1}$ commute. This is a relation among the twists about the five curves $c_{1,2},c_{2,3},c_{3,4},c_{1,3},c_{2,4}$ of the introduction. We do not know whether the twists in Question~\ref{q:square} satisfy relations of this kind.
\end{remark}

\section{Proofs of Theorem~\ref{thm:main} and Corollary~\ref{cor:free}}\label{sec:A}

In this section we prove Theorem~\ref{thm:main} with the Ping-Pong Lemma, using the witness lemmas of Section~\ref{sec:twists}. Then we prove Corollary~\ref{cor:free}, and we show that every group $\Z^{s_1}*\dots*\Z^{s_r}$ with $r\ge2$ is isomorphic to a subgroup generated by twists about relaxed curves. We use the Ping-Pong Lemma in the following form; see \cite{dlH,LS}.

\begin{lemma}[Ping-Pong Lemma]\label{lem:pingpong}
Let $G$ be a group acting on a set $X$, and let $G_1,\dots,G_r$, $r\ge2$, be infinite subgroups which generate $G$. Suppose that there are nonempty pairwise disjoint subsets $X_1,\dots,X_r$ of $X$ such that
\[
g(X_{q'})\subset X_q
\]
for all $q\ne q'$ and all $g\in G_q\setminus\{1\}$. Then $G$ is isomorphic to the free product $G_1*\cdots*G_r$.
\end{lemma}

\begin{proof}
The inclusions $G_q\to G$ define a homomorphism $G_1*\cdots*G_r\to G$, which is surjective since the subgroups $G_q$ generate $G$. We show that it is injective. A nontrivial element of the free product can be written as $w=g_1\cdots g_\ell$ with $\ell\ge1$, where $g_i\in G_{q_i}\setminus\{1\}$ and $q_i\ne q_{i+1}$ for each $i$. Suppose first that $q_1=q_\ell$, and write $q$ for this index. Let $j\ne q$. Applying $g_\ell,g_{\ell-1},\dots,g_1$ one after another, we get $w(X_j)\subset X_q$. Since $X_j$ is nonempty and disjoint from $X_q$, the image of $w$ in $G$ is not the identity.

Now suppose that $q_1\ne q_\ell$, and write $q=q_\ell$. Since $G_q$ is infinite, we can choose $h\in G_q\setminus\{1,g_\ell\}$. Then
\[
hwh^{-1}=h\,g_1\cdots g_{\ell-1}\,(g_\ell h^{-1})
\]
is as in the first case, so its image in $G$ is not the identity. Hence the image of $w$ is not the identity either.
\end{proof}

\begin{proof}[Proof of Theorem~\ref{thm:main}]
Let $G$ be the subgroup generated by the twists $t_c$, $c\in C$, where in the second part one twist of each pair is replaced by its square, and let $G_q$ be the subgroup generated by the elements coming from $P_q$. The twists about pairwise disjoint distinct curves generate a free abelian group whose rank is the number of curves \cite{FM}, so $G_q\cong\Z^{|P_q|}$; in particular $G_q$ is infinite. We apply Lemma~\ref{lem:pingpong} to the action of $G$ on $\mathcal S(D_n)$, with the sets
\[
X_q=\bigcup_{c\in P_q}Q_c .
\]
That is, a curve lies in $X_q$ if and only if it has both an upper and a lower arc joining the endpoint intervals of some curve of $P_q$. These sets are nonempty, since $c\in Q_c$. Curves in different subsets alternate, so the sets $X_q$ are pairwise disjoint by Lemma~\ref{lem:disjoint}. Let $q\ne q'$, let $g\in G_q\setminus\{1\}$, and let $\gamma\in X_{q'}$, that is, $\gamma\in Q_{c'}$ for some $c'\in P_{q'}$. We write
\[
g=\prod_{c\in P_q}t_c^{m_c}.
\]
If two curves $a,b\in P_q$ have a common endpoint interval, then $m_a$ or $m_b$ is even, since one of $t_a$ and $t_b$ was replaced by its square. The twists commute, so we may apply them in any order. We show that $g(\gamma)\in X_q$. First suppose that $m_a=m_b=0$ for every pair $a,b\in P_q$ with a common endpoint interval. We choose $c_0\in P_q$ with $m_{c_0}\ne0$, and we apply $t_{c_0}^{m_{c_0}}$ first. The curves $c'$ and $c_0$ alternate, so $t_{c_0}^{m_{c_0}}(\gamma)\in Q_{c_0}$ by Corollary~\ref{cor:alt}. Every other curve $c\in P_q$ with $m_c\ne0$ is disjoint from $c_0$, so it does not alternate with $c_0$, and the four endpoint intervals of $c$ and $c_0$ are distinct, since neither of them belongs to such a pair. So $t_c^{m_c}$ preserves $Q_{c_0}$ by Corollary~\ref{cor:nonalt}. Hence $g(\gamma)\in Q_{c_0}\subset X_q$.

Now suppose that $(m_a,m_b)\ne(0,0)$ for some pair $a,b\in P_q$ with a common endpoint interval. Since $m_a$ or $m_b$ is even, $(m_a,m_b)\ne\pm(1,-1)$. The curve $c'$ alternates with $a$ and with $b$, so $t_a^{m_a}t_b^{m_b}(\gamma)\in Q_a\cup Q_b$ by Lemma~\ref{lem:paired}. Every other curve $c$ of $P_q$ is disjoint from $a$ and $b$, so it does not alternate with them, and it has four distinct endpoint intervals with $a$ and with $b$, since neither $a$ nor $b$ has a common endpoint interval with a third curve. So $t_c^{m_c}$ preserves $Q_a$ and $Q_b$ by Corollary~\ref{cor:nonalt}. Hence $g(\gamma)\in Q_a\cup Q_b\subset X_q$. So $g(X_{q'})\subset X_q$. By Lemma~\ref{lem:pingpong}, $G\cong G_1*\cdots*G_r\cong\Z^{|P_1|}*\cdots*\Z^{|P_r|}$.
\end{proof}

\begin{remark}\label{rem:hypotheses}
In Theorem~\ref{thm:main}, the condition that curves in the same subset are disjoint and curves in different subsets alternate cannot be dropped. For instance, if a curve is disjoint from all the other curves, then its twist is central in the subgroup, so the subgroup is not a free product of two or more nontrivial groups (Example~\ref{ex:Z2Z}). The condition on endpoint intervals is used in Corollary~\ref{cor:nonalt}, since Lemma~\ref{lem:arcs} says nothing about an arc with an endpoint in a boundary interval. Without squaring a twist, the sets $X_q$ need not satisfy the hypotheses of Lemma~\ref{lem:pingpong} (Example~\ref{ex:D6}); see also Question~\ref{q:square}.

The second part of Theorem~\ref{thm:main} allows several pairs of curves with a common endpoint interval in the same subset. For instance, $c_{1,6},c_{2,6},c_{3,5},c_{4,5}$ are pairwise disjoint in $D_7$, the pairs $c_{1,6},c_{2,6}$ and $c_{3,5},c_{4,5}$ share $I_6$ and $I_5$ respectively, and $c_{5,7}$ alternates with all four, so
\[
\langle t_{1,6}^2,t_{2,6},t_{3,5}^2,t_{4,5},t_{5,7}\rangle\cong\Z^4*\Z.
\]
\end{remark}

\begin{remark}\label{rem:tits}
If every curve surrounds exactly two punctures, the square in Theorem~\ref{thm:main} is unnecessary: $t_{i,i+1}=\sigma_i^2$, and by a theorem of Collins \cite{Collins}, which proves a conjecture of Tits for braid groups (see also \cite{CrispParis,DLS}), the squares of the standard generators of $B_n$ generate the right-angled Artin group of their commutation graph. For instance $c_{1,2}$ and $c_{3,4}$ share the interval $I_2$ and both alternate with $c_{2,3}$, and $\langle t_{1,2},t_{2,3},t_{3,4}\rangle\cong\Z^2*\Z$ in $D_4$.
\end{remark}

\begin{proof}[Proof of Corollary~\ref{cor:free}]
Suppose that the curves pairwise alternate. Then they have pairwise distinct endpoint intervals. So we can apply Theorem~\ref{thm:main} to the partition of the $k\ge2$ curves into subsets with one element each, and the free product of the infinite cyclic groups $\langle t_{c_1}\rangle,\dots,\langle t_{c_k}\rangle$ is the free group freely generated by $t_{c_1},\dots,t_{c_k}$.

Conversely, suppose that two of the curves, say $c_1$ and $c_2$, do not alternate. They are disjoint by Remark~\ref{lem:intersection}, so $t_{c_1}$ and $t_{c_2}$ generate a subgroup isomorphic to $\Z^2$ \cite{FM}. Every subgroup of a free group is free \cite{MKS}, and $\Z^2$ is not free, so $\langle t_{c_1},\dots,t_{c_k}\rangle$ is not a free group.

For the last statement, pairwise alternating curves have $2k$ distinct endpoint intervals among the $n+1$ intervals $I_0,\dots,I_n$, so $2k\le n+1$. Now let $k=\lfloor(n+1)/2\rfloor$. Then $k\ge2$ since $n\ge3$, and $2k-1\le n$. Consider the curves $c_{i,i+k-1}$, $1\le i\le k$. Each of them surrounds $k$ punctures, and $2\le k\le n-1$ since $n\ge3$, so it is a relaxed curve in $D_n$. If $i<i'$, then $i'\le k\le i+k-1$, so $i<i'\le i+k-1<i'+k-1$, and the two curves alternate.
\end{proof}

Corollary~\ref{cor:free} gives $F_k$ from $k$ pairwise alternating relaxed curves in $D_{2k-1}$. The same disk also gives every group $\Z^{s_1}*\dots*\Z^{s_r}$ with $r\ge2$ and $s_1+\dots+s_r=k$.

\begin{figure}[ht]
\centering
\labellist
\hair 2pt
\pinlabel {{\ilab $I_{0}$}} at 15 56
\pinlabel {{\ilab $I_{1}$}} at 41 56
\pinlabel {{\ilab $I_{2}$}} at 70 56
\pinlabel {{\ilab $I_{3}$}} at 98 56
\pinlabel {{\ilab $I_{4}$}} at 130 56
\pinlabel {{\ilab $I_{5}$}} at 154 56
\pinlabel {{\ilab $I_{6}$}} at 185 56
\pinlabel {{\ilab $I_{7}$}} at 212 56
\pinlabel {{\scriptsize $c_{1,5}$}} at 33 107
\pinlabel {{\scriptsize $c_{2,4}$}} at 52 90
\pinlabel {{\scriptsize \textcolor{blue!70!black}{$c_{3,7}$}}} at 195 107
\pinlabel {{\scriptsize \textcolor{blue!70!black}{$c_{4,6}$}}} at 175 90
\endlabellist
\includegraphics{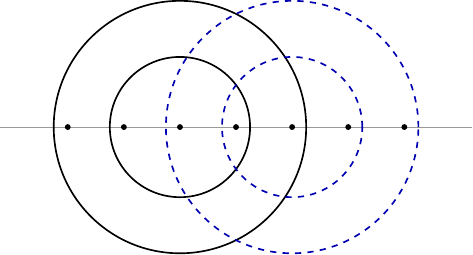}
\caption{The curves of Corollary~\ref{cor:ranks} for $(s_1,s_2)=(2,2)$ in $D_7$. The curves $c_{1,5},c_{2,4}$ of $P_1$ (solid) start in $I_0,I_1$ and end in $I_4,I_5$, and the curves $c_{3,7},c_{4,6}$ of $P_2$ (dashed) start in $I_2,I_3$ and end in $I_6,I_7$. Curves in the same subset are nested, and curves in different subsets alternate.}
\label{fig:ranks}
\end{figure}

\begin{corollary}\label{cor:ranks}
Let $r\ge2$, let $s_1,\dots,s_r$ be positive integers and let $k=s_1+\dots+s_r$. Then there are $k$ relaxed curves in $D_{2k-1}$ whose twists generate a subgroup isomorphic to $\Z^{s_1}*\dots*\Z^{s_r}$.
\end{corollary}

\begin{proof}
We choose $k$ curves, one starting in each of the intervals $I_0,\dots,I_{k-1}$ and one ending in each of the intervals $I_k,\dots,I_{2k-1}$. We group $I_0,\dots,I_{k-1}$ from left to right into blocks of $s_1,\dots,s_r$ consecutive intervals, and we group $I_k,\dots,I_{2k-1}$ in the same way. The curves starting in the $q$-th block of the first group end in the $q$-th block of the second group, in reverse order, and they form the subset $P_q$; see Figure~\ref{fig:ranks}. All endpoint intervals are distinct. Curves in the same subset are nested, so they are disjoint. If $q<q'$, then a curve of $P_q$ starts and ends to the left of the start and the end of a curve of $P_{q'}$, and since all curves start to the left of all ends, the two curves alternate. Each curve of $P_q$ surrounds between $k-s_q+1\ge2$ and $k+s_q-1\le2k-2$ punctures, so it is relaxed. The result follows from Theorem~\ref{thm:main}.
\end{proof}

\section{Examples}\label{sec:examples}

In this section we give examples of Theorem~\ref{thm:main}. We start with small examples in $D_5$. Next we give examples with a shared endpoint interval. Theorem~\ref{thm:humphries} does not apply to the families of Examples~\ref{ex:D6} and \ref{ex:D7} for any representatives with finitely many pairwise transverse intersections (Remark~\ref{remark:multiple-points}). It does apply to the family of Example~\ref{ex:F3}, but only for representatives which are not in general position (Propositions~\ref{prop:pencil} and \ref{prop:humphries-free}).

\begin{example}[$F_3$]\label{ex:F3}
\emergencystretch=1em
The curves $c_{1,3},c_{2,4},c_{3,5}$ in $D_5$ pairwise alternate (Figure~\ref{fig:F3}), so by Corollary~\ref{cor:free}
\[
\langle t_{1,3},t_{2,4},t_{3,5}\rangle\cong F_3 .
\]
The same corollary shows that $D_n$ contains $k$ pairwise alternating relaxed curves only if ${2k\le n+1}$, so $D_5$ is the smallest disk which contains three of them.
Proposition~\ref{prop:humphries-free} gives representatives through two common points satisfying Humphries' hypothesis, so this free group also follows from his theorem.
\end{example}

\begin{figure}[ht]
\centering
\labellist
\hair 2pt
\pinlabel {{\ilab $I_{0}$}} at 22 41
\pinlabel {{\ilab $I_{1}$}} at 55 41
\pinlabel {{\ilab $I_{2}$}} at 92 41
\pinlabel {{\ilab $I_{3}$}} at 133 41
\pinlabel {{\ilab $I_{4}$}} at 170 41
\pinlabel {{\ilab $I_{5}$}} at 207 41
\pinlabel {{\scriptsize $c_{1,3}$}} at 32 85
\pinlabel {{\scriptsize \textcolor{blue!70!black}{$c_{2,4}$}}} at 115 100
\pinlabel {{\scriptsize \textcolor{red!75!black}{$c_{3,5}$}}} at 198 85
\endlabellist
\includegraphics{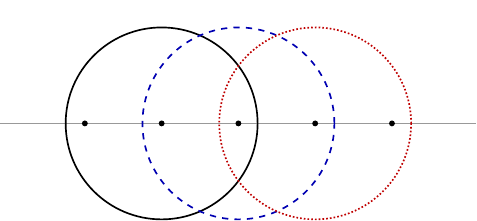}
\caption{The curves $c_{1,3},c_{2,4},c_{3,5}$ in $D_5$.}
\label{fig:F3}
\end{figure}

\begin{example}[$\Z^2*\Z$]\label{ex:Z2Z}
In $D_5$ let $P_1=\{c_{1,4},c_{2,3}\}$ and $P_2=\{c_{3,5}\}$. The six endpoint intervals are distinct, and $c_{3,5}$ alternates with $c_{1,4}$ and with $c_{2,3}$. So $\langle t_{1,4},t_{2,3},t_{3,5}\rangle\cong\Z^2*\Z$ by Theorem~\ref{thm:main}. This is the case $(s_1,s_2)=(2,1)$ of Corollary~\ref{cor:ranks}. Now consider $c_{1,4},c_{1,2},c_{3,5}$ instead. The curve $c_{1,2}$ is disjoint from $c_{1,4}$ and from $c_{3,5}$, so $t_{1,2}$ is central in the subgroup. Since $\langle t_{1,4},t_{3,5}\rangle\cong F_2$ by Corollary~\ref{cor:free} and $F_2$ has trivial centre, the subgroup is isomorphic to $F_2\times\Z$, and not to $\Z^2*\Z$. In this case there is no partition as in Theorem~\ref{thm:main}, since $c_{1,4}$ and $c_{3,5}$ intersect.
\end{example}

\begin{example}\label{ex:D6}
In $D_6$ let $P_1=\{c_{1,4},c_{2,4}\}$, $P_2=\{c_{3,5}\}$ and $P_3=\{c_{4,6}\}$. The curves $c_{1,4},c_{3,5},c_{4,6}$ pairwise alternate, so their twists generate $F_3$ by Corollary~\ref{cor:free}. The curve $c_{2,4}$ lies inside $c_{1,4}$, and it alternates with $c_{3,5}$ and $c_{4,6}$. However, $c_{1,4}$ and $c_{2,4}$ share the endpoint interval $I_4$, so the first part of Theorem~\ref{thm:main} does not apply. If we square $t_{1,4}$, then the second part gives
\[
\langle t_{1,4}^2,\,t_{2,4},\,t_{3,5},\,t_{4,6}\rangle\cong\Z^2*\Z*\Z .
\]
Without the square, the sets of Section~\ref{sec:A} do not work for this family. Let $X_1=Q_{1,4}\cup Q_{2,4}$, $X_2=Q_{3,5}$ and $X_3=Q_{4,6}$. We show that $t_{1,4}t_{2,4}^{-1}(X_2)\not\subset X_1$.

Let $\gamma=c_{3,5}\in X_2$, with chain $\ttarc{I_2I_5}\ssarc{I_5I_2}$, and write $A=t_{1,4}$ and $B=t_{2,4}$. The endpoint intervals of $c_{1,4}$ are $I_0$ and $I_4$. Both arcs of $\gamma$ cross $c_{1,4}$, with inner endpoint in $I_2$ and outer endpoint in $I_5$. For $c_{1,4}$ we have $T^1_U=(0,4)$ and $T^1_D=(4,0)$, so Proposition~\ref{prop:twist} gives
\[
w(A\gamma)=\ttarc{I_2I_4}\ssarc{I_4I_0}\ttarc{I_0I_5}\ssarc{I_5I_4}\ttarc{I_4I_0}\ssarc{I_0I_2}.
\]
The endpoint intervals of $c_{2,4}$ are $I_1$ and $I_4$. The arcs of $A\gamma$ which cross $c_{2,4}$ are $\ttarc{I_4I_2}$ and $\ssarc{I_0I_2}$, both with inner endpoint in $I_2$. For $c_{2,4}$ we have $T^{-1}_U=(4,1)$ and $T^{-1}_D=(1,4)$. So $B^{-1}$ replaces $\ttarc{I_4I_2}$ by $\ttarc{I_4I_4}\ssarc{I_4I_1}\ttarc{I_1I_2}$, and $\ssarc{I_0I_2}$ by $\ssarc{I_0I_1}\ttarc{I_1I_4}\ssarc{I_4I_2}$. This gives
\[
\ttarc{I_2I_1}\ssarc{I_1I_4}\ttarc{I_4I_4}\ssarc{I_4I_0}\ttarc{I_0I_5}\ssarc{I_5I_4}\ttarc{I_4I_0}\ssarc{I_0I_1}\ttarc{I_1I_4}\ssarc{I_4I_2}.
\]
Move~1 replaces $\ssarc{I_1I_4}\ttarc{I_4I_4}\ssarc{I_4I_0}$ by $\ssarc{I_1I_0}$, and we get
\[
w(B^{-1}A\gamma)=\ttarc{I_2I_1}\ssarc{I_1I_0}\ttarc{I_0I_5}\ssarc{I_5I_4}\ttarc{I_4I_0}\ssarc{I_0I_1}\ttarc{I_1I_4}\ssarc{I_4I_2}.
\]
This chain is reduced. It contains the upper arcs $\ttarc{I_0I_4}$ and $\ttarc{I_1I_4}$, but neither of the lower arcs $\ssarc{I_0I_4}$ and $\ssarc{I_1I_4}$. So $B^{-1}A\gamma$ lies neither in $Q_{1,4}$ nor in $Q_{2,4}$, that is, $B^{-1}A\gamma\notin X_1$. Since $A$ and $B$ commute, this proves that $t_{1,4}t_{2,4}^{-1}(X_2)\not\subset X_1$. Similarly, for $t_{1,4}^{-1}t_{2,4}$ we find
\[
w(A^{-1}B\gamma)=\ttarc{I_0I_1}\ssarc{I_1I_2}\ttarc{I_2I_4}\ssarc{I_4I_1}\ttarc{I_1I_0}\ssarc{I_0I_4}\ttarc{I_4I_5}\ssarc{I_5I_0},
\]
which contains the lower arcs $\ssarc{I_0I_4}$ and $\ssarc{I_1I_4}$ but neither of the upper arcs $\ttarc{I_0I_4}$ and $\ttarc{I_1I_4}$.
\end{example}

\begin{example}\label{ex:Z2Z2}
In $D_5$ let $P_1=\{c_{1,2},c_{1,3}\}$ and $P_2=\{c_{2,4},c_{2,5}\}$. The curves of $P_1$ share the interval $I_0$, the curves of $P_2$ share the interval $I_1$, and each curve of $P_1$ alternates with each curve of $P_2$. So $\langle t_{1,2}^2,t_{1,3},t_{2,4}^2,t_{2,5}\rangle\cong\Z^2*\Z^2$ by the second part of Theorem~\ref{thm:main}.
\end{example}

\begin{example}\label{ex:D7}
In $D_7$ let $P_1=\{c_{1,5},c_{2,4}\}$, $P_2=\{c_{3,6}\}$ and $P_3=\{c_{4,7}\}$ (Figure~\ref{fig:D7}). Their endpoint intervals are
\[
I_0,I_5\ \text{for } c_{1,5},\qquad I_1,I_4\ \text{for } c_{2,4},\qquad I_2,I_6\ \text{for } c_{3,6},\qquad I_3,I_7\ \text{for } c_{4,7},
\]
so all eight endpoint intervals are distinct. The curve $c_{2,4}$ lies inside $c_{1,5}$, and every other pair of curves alternates. So Theorem~\ref{thm:main} gives
\[
\langle t_{1,5},t_{2,4},t_{3,6},t_{4,7}\rangle\cong\Z^2*\Z*\Z .
\]
Unlike in Example~\ref{ex:D6}, no square is needed here, since $c_{1,5}$ and $c_{2,4}$ have no common endpoint interval. This family and the family of Example~\ref{ex:D6} are the smallest of their kind. A family of relaxed curves which is partitioned into a pair and two curves, as in the second part of Theorem~\ref{thm:main}, needs at least seven endpoint intervals. Indeed, the two curves which form subsets by themselves have four distinct endpoint intervals, and the pair has at least three further ones. So such a family exists only if $n\ge6$. If the endpoint intervals are pairwise distinct, then we need eight of them, so $n\ge7$.
\end{example}

\begin{figure}[ht]
\centering
\labellist
\hair 2pt
\pinlabel {{\ilab $I_{0}$}} at 14 62
\pinlabel {{\ilab $I_{1}$}} at 44 62
\pinlabel {{\ilab $I_{2}$}} at 78 62
\pinlabel {{\ilab $I_{3}$}} at 108 62
\pinlabel {{\ilab $I_{4}$}} at 138 62
\pinlabel {{\ilab $I_{5}$}} at 168 62
\pinlabel {{\ilab $I_{6}$}} at 201 62
\pinlabel {{\ilab $I_{7}$}} at 231 62
\pinlabel {{\scriptsize $c_{1,5}$}} at 33 118
\pinlabel {{\scriptsize $c_{2,4}$}} at 56 99
\pinlabel {{\scriptsize \textcolor{blue!70!black}{$c_{3,6}$}}} at 105 119
\pinlabel {{\scriptsize \textcolor{red!75!black}{$c_{4,7}$}}} at 217 109
\endlabellist
\includegraphics{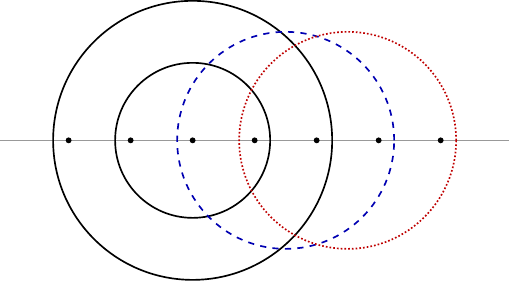}
\caption{The curves $c_{1,5},c_{2,4},c_{3,6},c_{4,7}$ in $D_7$; $c_{2,4}$ is nested in $c_{1,5}$ and every other pair alternates.}
\label{fig:D7}
\end{figure}

\section{Comparison with Theorem~\ref{thm:humphries}}\label{sec:disks}
In this section we compare our criteria with Theorem~\ref{thm:humphries}. For the families of Examples~\ref{ex:D6} and~\ref{ex:D7}, Theorem~\ref{thm:main} gives free product decompositions, with $t_{1,4}$ replaced by its square in Example~\ref{ex:D6} and with the twists themselves in Example~\ref{ex:D7}. We show that neither family has representatives with finitely many pairwise transverse intersections satisfying Humphries' complementary region hypothesis. We then discuss the free group case. We use semicircular arcs for the coordinate calculations. For the comparison with Theorem~\ref{thm:humphries}, we also consider other drawings within the same isotopy classes, since its hypothesis depends on the complementary regions. We say that curves in $D_n$ are \emph{in general position} if any two of them intersect transversally in finitely many points, and no three of them have a common point. The curves in Figure~\ref{fig:F3} are in general position. Any curves can be moved into general position by a small isotopy, but such an isotopy can create complementary disks. For instance, resolving a triple point at which the curves meet pairwise transversally can create a triangle without punctures. Thus passing to general position need not preserve the hypothesis of Theorem~\ref{thm:humphries}. The statement of Theorem~\ref{thm:humphries} in \cite{Humphries} does not assume general position, and we apply it as it is stated. We first consider representatives in general position. If such curves intersect many times, then the complement of their union has more components than there are punctures, so one of the components is a disk without punctures. The next lemma makes this count precise.

\begin{lemma}\label{lem:faces}
Let $\gamma_1,\dots,\gamma_k$, $k\ge2$, be curves in $D_n$ in general position whose union $K$ is connected, and let $N$ be the number of intersection points of these curves. Then $D_n\setminus K$ has one component which contains the boundary of $D_n$, and $N+1$ other components. Each of these other components is an open disk with finitely many punctures. In particular, if $N\ge n$, then some component of $D_n\setminus K$ is a disk without punctures.
\end{lemma}

\begin{proof}
We fill in the punctures and regard $D_n$ as a subset of the sphere. Then $K$ is a connected graph in the sphere whose vertices are the $N$ intersection points and whose edges are the arcs of the curves between them; see Figure~\ref{fig:faces}. Each vertex has degree four, so $K$ has $2N$ edges. By Euler's formula, $K$ has $2-N+2N=N+2$ faces, and each face is an open disk, since $K$ is connected. One face contains the complement of $D_n$. At most $n$ of the other $N+1$ faces contain punctures, so if $N\ge n$, one of them contains no puncture.
\end{proof}

Note that any representatives of two alternating relaxed curves intersect at least twice. For representatives in general position, the number $N$ in Lemma~\ref{lem:faces} is therefore at least twice the number of alternating pairs. We apply this observation to the families of Examples~\ref{ex:D6} and~\ref{ex:D7}.

\begin{example}\label{prop:D6disk}
Let $a=c_{1,4}$, $b=c_{2,4}$, $c=c_{3,5}$ and $d=c_{4,6}$ be relaxed curves in $D_6$, represented in general position.
The five pairs $\{a,c\},\{a,d\},\{b,c\},\{b,d\},\{c,d\}$ have geometric intersection number two, so their union $K$ is connected and they have $N\ge10$ intersection points. By Lemma~\ref{lem:faces} there are at least $11$ components other than the one containing the boundary. At most $6$ of them contain punctures, so at least five are disks without punctures (Figure~\ref{fig:faces}). Hence no representatives in general position of this family satisfy the complementary region hypothesis of Theorem~\ref{thm:humphries}.
\end{example}

\begin{figure}[h!]
\centering
\labellist
\hair 2pt
\pinlabel {{\fontsize{4.5}{5}\selectfont $v_{1}$}} at 96 120
\pinlabel {{\fontsize{4.5}{5}\selectfont $v_{2}$}} at 121 129
\pinlabel {{\fontsize{4.5}{5}\selectfont $v_{3}$}} at 122 102
\pinlabel {{\fontsize{4.5}{5}\selectfont $v_{4}$}} at 143 108
\pinlabel {{\fontsize{4.5}{5}\selectfont $v_{5}$}} at 164 120
\pinlabel {{\fontsize{4.5}{5}\selectfont $v_{6}$}} at 96 17
\pinlabel {{\fontsize{4.5}{5}\selectfont $v_{7}$}} at 121 8
\pinlabel {{\fontsize{4.5}{5}\selectfont $v_{8}$}} at 122 35
\pinlabel {{\fontsize{4.5}{5}\selectfont $v_{9}$}} at 143 30
\pinlabel {{\fontsize{4.5}{5}\selectfont $v_{10}$}} at 164 17
\pinlabel {{\scriptsize $a$}} at 19 122
\pinlabel {{\scriptsize $b$}} at 57 108
\pinlabel {{\scriptsize \textcolor{blue!70!black}{$c$}}} at 149 129
\pinlabel {{\scriptsize \textcolor{blue!70!black}{$d$}}} at 210 105
\endlabellist
\includegraphics{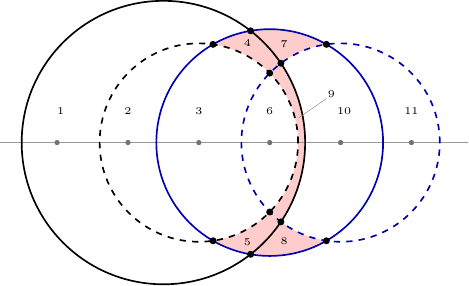}
\caption{The relaxed curves $a=c_{1,4}$, $b=c_{2,4}$, $c=c_{3,5}$ and $d=c_{4,6}$ in $D_6$, represented in general position as in Example~\ref{prop:D6disk}. They meet in $N=10$ points $v_1,\dots,v_{10}$, and the complement of their union has $11$ components other than the one containing the boundary. The components $4,5,7,8,9$ contain no puncture.}
\label{fig:faces}
\end{figure}

\begin{example}\label{prop:D7disk}
Let $a=c_{1,5}$, $b=c_{2,4}$, $c=c_{3,6}$ and $d=c_{4,7}$ be relaxed curves in $D_7$, represented in general position. The five pairs other than $\{a,b\}$ have geometric intersection number two, so their union $K$ is connected and they have $N\ge10$ intersection points. By Lemma~\ref{lem:faces} there are at least $11$ components other than the one containing the boundary. At most $7$ of them contain punctures, so at least four are disks without punctures. Hence no representatives in general position of this family satisfy the complementary region hypothesis of Theorem~\ref{thm:humphries}.
\end{example}

\begin{remark}\label{remark:multiple-points}
For each of the families in Examples~\ref{prop:D6disk} and~\ref{prop:D7disk}, every choice of representatives with finitely many pairwise transverse intersections has a complementary disk without punctures, even when several curves pass through the same point. Consider either family, with the notation $a,b,c,d$ of the corresponding example. In both families, $\iota(a,b)=0$, while every other pair has geometric intersection number two. If two representatives are not in minimal position, the bigon criterion gives an unpunctured bigon \cite[Proposition~1.7]{FM}. The portions of each remaining curve inside the bigon are arcs with endpoints on its boundary, since an essential curve cannot lie entirely inside an unpunctured disk. These arcs cut the bigon into disk regions. Thus an unpunctured complementary disk exists. Suppose instead that the representatives are pairwise minimal. Since $\iota(a,b)=0$, the representatives of $a$ and $b$ are disjoint, so at most three curves pass through an intersection point. Let $m_v$ be the number of curves through an intersection point $v$. The union is a connected planar graph with $E=\sum_v m_v$ edges, and hence has
\[
B=1+\sum_v(m_v-1)
\]
bounded faces. The five intersecting pairs meet twice each, so
\[
10=\sum_v\binom{m_v}{2}\le\frac32\sum_v(m_v-1).
\]
Consequently $B\ge23/3$, hence $B\ge8$. Since the disk has only six or seven punctures, some bounded face is an unpunctured disk.
\end{remark} We next apply Lemma~\ref{lem:faces} to the free group case.

\begin{proposition}\label{prop:pencil}
Let $c_{i_1,j_1},\dots,c_{i_k,j_k}$, $k\ge2$, be pairwise alternating relaxed curves in $D_n$ with $n\le k(k-1)$. Then for any representatives of these curves in general position, some component of the complement of their union is a disk without punctures. In particular, no such representatives satisfy the complementary region hypothesis of Theorem~\ref{thm:humphries}.
\end{proposition}

\begin{proof}
Any two of the curves have geometric intersection number two, so their union is connected and there are at least $2\binom k2=k(k-1)\ge n$ intersection points. Lemma~\ref{lem:faces} gives a disk without punctures.
\end{proof}

For instance, this applies to $c_{1,3},c_{2,4},c_{3,5}$ in $D_5$ when they are represented in general position. For pairwise alternating relaxed curves, the free group conclusion also follows from Humphries' theorem, since we can choose representatives satisfying its hypothesis.

\begin{proposition}\label{prop:humphries-free}
A family of $k\ge2$ pairwise alternating relaxed curves can be drawn with two common intersection points and no unpunctured complementary disk. Hence Theorem~\ref{thm:humphries}, applied to these representatives, gives the free group conclusion of Corollary~\ref{cor:free}.
\end{proposition}
\begin{proof}
Order the left endpoints as $\ell_1<\cdots<\ell_k$. Alternation implies that the right endpoints occur in the same order and that
\[
\ell_1<\cdots<\ell_k<r_1<\cdots<r_k.
\]
The $2k$ endpoints lie in distinct intervals of $H$. Draw all upper arcs through one point $P\in D^U$ and all lower arcs through one point $Q\in D^D$, with no other intersections and with each pair transverse at $P$ and $Q$; see Figure~\ref{fig:pencil} for $k=3$. Such a drawing can be obtained from circles through two fixed points on opposite sides of $H$ by a homeomorphism preserving $H$ and its order. Each curve still joins the prescribed endpoint intervals, so it is in the required isotopy class.

The resulting graph has two vertices and $2k$ edges, and hence $2k-1$ bounded faces. Each bounded face meets $H$ in the interval between two consecutive endpoints. Because the endpoints lie in distinct endpoint intervals, each such segment contains a puncture. The remaining component is adjacent to $\partial D_n$ and is not a disk. The complete partition consists of singletons, so Theorem~\ref{thm:humphries} applies.
\end{proof}

\begin{figure}[htbp]
\centering
\labellist
\hair 2pt
\pinlabel {{\ilab $I_{0}$}} at 25 63
\pinlabel {{\ilab $I_{1}$}} at 57 63
\pinlabel {{\ilab $I_{2}$}} at 96 63
\pinlabel {{\ilab $I_{3}$}} at 132 63
\pinlabel {{\ilab $I_{4}$}} at 172 63
\pinlabel {{\ilab $I_{5}$}} at 203 63
\pinlabel {{\small $P$}} at 112 128
\pinlabel {{\small $Q$}} at 111 11
\pinlabel {{\scriptsize $H$}} at 223 68
\pinlabel {{\scriptsize $c_{1,3}$}} at 36 111
\pinlabel {{\scriptsize \textcolor{blue!70!black}{$c_{2,4}$}}} at 87 129
\pinlabel {{\scriptsize \textcolor{red!75!black}{$c_{3,5}$}}} at 190 111
\endlabellist
\includegraphics{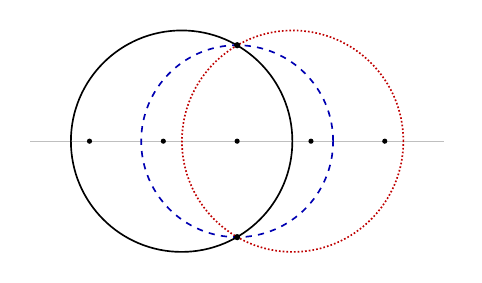}
\caption{The relaxed curves $c_{1,3},c_{2,4},c_{3,5}$ in $D_5$ through two common points $P,Q$. Each of the five bounded faces contains a puncture.}
\label{fig:pencil}
\end{figure}

These results raise the following questions.

\begin{question}\label{q:square}
Can the square in Theorem~\ref{thm:main} be removed, and can the theorem be extended to subsets in which a curve has a common endpoint interval with two or more other curves, such as $c_{1,4},c_{2,4},c_{3,4}$?
\end{question}

For the curves of Example~\ref{ex:D6}, Theorem~\ref{thm:humphries} does not apply (Remark~\ref{remark:multiple-points}), and the sets of Section~\ref{sec:A} do not work without the square. We expect that the square is not needed.

\begin{conjecture}
For the curves of Example~\ref{ex:D6}, $\langle t_{1,4},t_{2,4},t_{3,5},t_{4,6}\rangle\cong\Z^2*\Z*\Z$.
\end{conjecture}

\begin{question}
Do witness arcs give criteria for curves which cross $H$ more than twice?
\end{question}

\section{The algorithm}\label{sec:algorithm}
\noindent The Dynnikov coordinate function  $\rho\colon\mathcal S(D_n)\to\Z^{2n-4}\setminus\{0\}$ is a bijection \cite{Dynnikov} (see Appendix~\ref{sec:coords}).  A relaxed curve $c_{i,j}$ has Dynnikov coordinates 
\begin{equation}\label{eq:relaxedcoords}
\rho(c_{i,j})=(0,\dots,0;\,b_1,\dots,b_{n-2}),
\end{equation}
where $b_{i-1}=-1$ if $i>1$, $b_{j-1}=1$ if $j<n$, and all other entries of $b$ are $0$. Hence a curve is relaxed if and only if its Dynnikov coordinates are of the form \eqref{eq:relaxedcoords} for some $1\le i<j\le n$ with $j-i\le n-2$, and then it is $c_{i,j}$. Together with Theorem~\ref{thm:main} this gives the following algorithm.

\begin{algorithm}\label{alg:witness}
Let $(a^{(1)};b^{(1)}),\dots,(a^{(k)};b^{(k)})\in\Z^{2n-4}$ be the Dynnikov coordinates of pairwise distinct curves $\gamma_1,\dots,\gamma_k$ in $D_n$, $k\ge2$. The following algorithm returns a free abelian group or a free product of free abelian groups generated by the Dehn twists about $\gamma_1,\dots,\gamma_k$, where some of the twists may be replaced by their squares, or \emph{the criteria do not apply}.

\smallskip
\noindent\textbf{Step 1.} \emph{Identify the relaxed curves.}
If some $(a^{(h)};b^{(h)})$ is not of the form \eqref{eq:relaxedcoords}, return \emph{the criteria do not apply}. Otherwise write $\gamma_h=c_{i_h,j_h}$ as in \eqref{eq:relaxedcoords}.

\smallskip
\noindent\textbf{Step 2.} \emph{Check which pairs alternate.}
For each pair $h<h'$, the curves $\gamma_h$ and $\gamma_{h'}$ alternate if
\[
i_h<i_{h'}\le j_h<j_{h'}
\qquad\text{or}\qquad
i_{h'}<i_h\le j_{h'}<j_h,
\]
and are disjoint otherwise. If every pair alternates, return $F_k$, freely generated by $t_{i_1,j_1},\dots,t_{i_k,j_k}$, and stop. Otherwise the subgroup generated by these twists is not free, by Corollary~\ref{cor:free}.

\smallskip
\noindent\textbf{Step 3.} \emph{Form the subsets.}
Let $P_1$ consist of $\gamma_1$ and all curves which do not alternate with it. Take the first curve not in $P_1$, and form $P_2$ in the same way from the remaining curves. Repeat until every curve lies in some $P_q$. If two curves in the same subset alternate, or two curves in different subsets do not alternate, return \emph{the criteria do not apply}. If there is only one subset, return $\Z^k$, generated by $t_{i_1,j_1},\dots,t_{i_k,j_k}$, and stop.

\smallskip
\noindent\textbf{Step 4.} \emph{Check the endpoint intervals.}
For each subset $P_q$, list the pairs of curves $c_{i_h,j_h}$, $c_{i_{h'},j_{h'}}$ in $P_q$ for which the sets $\{i_h-1,j_h\}$ and $\{i_{h'}-1,j_{h'}\}$ intersect. If some curve belongs to two such pairs, return \emph{the criteria do not apply}.

\smallskip
\noindent\textbf{Step 5.} \emph{Choose the generators.}
For each pair found in Step~4, replace the twist about the curve with the smaller index $h$ by its square.

\begin{samepage}
\smallskip
\noindent\textbf{Step 6.} \emph{Return the free product.}
Return
\[
\Z^{|P_1|}*\cdots*\Z^{|P_r|},
\]
where $\Z^{|P_q|}$ is generated by the twists of Step~5 for the curves in $P_q$.
\end{samepage}
\end{algorithm}

In Step~1, a relaxed curve has $a^{(h)}=0$, so this entry is checked first. The inequalities in Step~2 are the definition of alternating curves in Section~\ref{sec:introduction}, and relaxed curves which do not alternate are disjoint by Remark~\ref{lem:intersection}. In Step~3, if a partition as in Theorem~\ref{thm:main} exists, then two curves lie in the same subset exactly when they do not alternate, so Step~3 finds it. If there is only one subset, the curves are pairwise disjoint, and their twists generate $\Z^k$ \cite{FM}. In Step~4 the sets $\{i_h-1,j_h\}$ consist of the indices of the endpoint intervals of $c_{i_h,j_h}$, so the pairs listed are those with a common endpoint interval; curves in different subsets alternate and so have no common endpoint interval, which is why only pairs in the same subset are listed. The output is correct by Theorem~\ref{thm:main}. Note that the algorithm may also return \emph{the criteria do not apply} in cases where the twists do generate a free product, since Theorem~\ref{thm:main} gives sufficient conditions; see Question~\ref{q:square}.

\appendix
\numberwithin{theorem}{subsection}
\section{Dynnikov coordinates and $\pi_1$-train track coordinates}\label{sec:dynnikov-coordinates}

Here we recall Dynnikov coordinates and the formulae of \cite{HY09,YG19} which relate them to the measures $\mu(\ttarc{I_sI_t})$ and $\mu(\ssarc{I_sI_t})$ of Section~\ref{sec:chains}.

\subsection{The Dynnikov coordinate system}\label{sec:coords}
We now define the Dynnikov coordinate function $\rho\colon\mathcal S(D_n)\to\Z^{2n-4}\setminus\{0\}$ of Remark~\ref{prop:dynnikov}, following \cite{Y11,HY09}. Construct the arcs $\alpha_r$ ($1\le r\le 2n$) and $\beta_s$ ($1\le s\le n-1$) in $D_n$ as depicted in Figure~\ref{fig:dynnikov}. For $1\le i\le n$, the arcs $\alpha_{2i-1}$ and $\alpha_{2i}$ join the $i$-th puncture to the boundary in $D^U$ and in $D^D$, respectively, and the arc $\beta_s$ has both endpoints on the boundary and crosses $H$ once, in $I_s$. Given $\mathcal L\in\mathcal S(D_n)$, let $L$ be a representative of $\mathcal L$ which intersects each of these arcs minimally. Write $\alpha_r$ (respectively $\beta_s$) for the number of intersections of $L$ with the arc $\alpha_r$ (respectively the arc $\beta_s$).

\begin{figure}[ht]
\centering
\labellist
\hair 2pt
\pinlabel {{\rotatebox{90}{\ilab $\alpha_1$}}} [l] at 26 115
\pinlabel {{\rotatebox{90}{\ilab $\alpha_2$}}} [l] at 26 76
\pinlabel {{\rotatebox{90}{\ilab $\alpha_3$}}} [l] at 62 125
\pinlabel {{\rotatebox{90}{\ilab $\alpha_4$}}} [l] at 62 66
\pinlabel {{\rotatebox{90}{\ilab $\alpha_{2i-3}$}}} [l] at 138 134
\pinlabel {{\rotatebox{90}{\ilab $\alpha_{2i-2}$}}} [l] at 138 56
\pinlabel {{\rotatebox{90}{\ilab $\alpha_{2i-1}$}}} [l] at 175 136
\pinlabel {{\rotatebox{90}{\ilab $\alpha_{2i}$}}} [l] at 175 55
\pinlabel {{\rotatebox{90}{\ilab $\alpha_{2i+1}$}}} [l] at 211 136
\pinlabel {{\rotatebox{90}{\ilab $\alpha_{2i+2}$}}} [l] at 211 55
\pinlabel {{\rotatebox{90}{\ilab $\alpha_{2n-3}$}}} [l] at 319 125
\pinlabel {{\rotatebox{90}{\ilab $\alpha_{2n-2}$}}} [l] at 319 66
\pinlabel {{\rotatebox{90}{\ilab $\alpha_{2n-1}$}}} [r] at 352 115
\pinlabel {{\rotatebox{90}{\ilab $\alpha_{2n}$}}} [r] at 352 76
\pinlabel {$\cdots$} at 99 95
\pinlabel {$\cdots$} at 263 95
\pinlabel {{\ilab $\beta_1$}} [tr] at 17 13
\pinlabel {{\ilab $\beta_{i-1}$}} [br] at 122 184
\pinlabel {{\ilab $\beta_i$}} [bl] at 224 184
\pinlabel {{\ilab $\beta_{n-1}$}} [tl] at 360 13
\endlabellist
\includegraphics[scale=0.9]{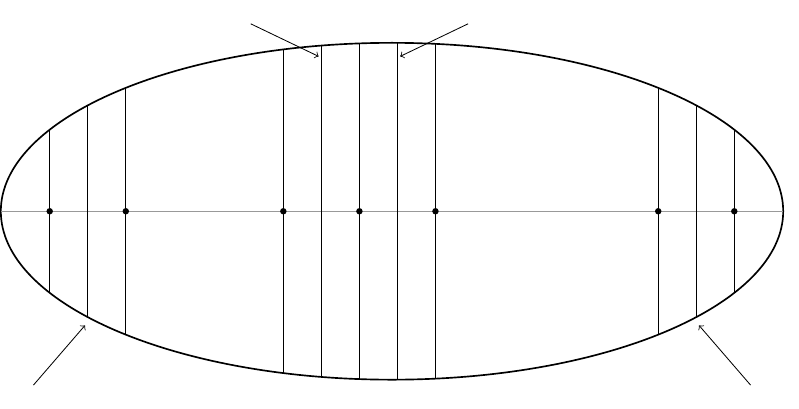}
\caption{The arcs $\alpha_r$ and $\beta_s$ in $D_n$.}
\label{fig:dynnikov}
\end{figure}

The \emph{Dynnikov coordinate function} $\rho\colon\mathcal S(D_n)\to\Z^{2n-4}\setminus\{0\}$ is defined by
\[
\rho(\mathcal L)=(a;b)=(a_1,\dots,a_{n-2};b_1,\dots,b_{n-2}),
\]
where
\[
a_i=\frac{\alpha_{2i+2}-\alpha_{2i+1}}{2}\qquad\text{and}\qquad b_i=\frac{\beta_i-\beta_{i+1}}{2}
\]
for $1\le i\le n-2$. We also write $b_0=-\beta_1/2$, $b_{n-1}=\beta_{n-1}/2$ and $x^+=\max(x,0)$.

\subsection{Inversion of Dynnikov coordinates}\label{sec:inversion}
The intersection numbers $\alpha_r$ and $\beta_s$, and hence $\mathcal L$, can be recovered from the Dynnikov coordinates $(a;b)$ by the following formulae \cite{HY09,YG19}. For $1\le s\le n-1$,
\begin{equation}\label{eq:beta}
\beta_s=2\max_{1\le k\le n-2}\Bigl(|a_k|+b_k^++\sum_{q=1}^{k-1}b_q\Bigr)-2\sum_{q=1}^{s-1}b_q .
\end{equation}
\begin{samepage}
For $3\le r\le2n-2$,
\begin{equation}\label{eq:alpha}
\alpha_r=
\begin{cases}
(-1)^r a_{\lceil r/2\rceil-1}+\dfrac{\beta_{\lceil r/2\rceil-1}}{2}, & \text{if } b_{\lceil r/2\rceil-1}\ge0,\\[6pt]
(-1)^r a_{\lceil r/2\rceil-1}+\dfrac{\beta_{\lceil r/2\rceil}}{2}, & \text{if } b_{\lceil r/2\rceil-1}\le0,
\end{cases}
\end{equation}
where $\lceil x\rceil$ denotes the smallest integer which is not less than $x$. For the first and the last puncture $\alpha_1=\alpha_2=\beta_1/2$ and $\alpha_{2n-1}=\alpha_{2n}=\beta_{n-1}/2$.
\end{samepage}

\subsection{From intersection numbers to $\pi_1$-train track coordinates}\label{sec:toarcs}
The measures below are taken from a representative which is minimal with respect to $H$, as in Section~\ref{sec:chains}, and we write $\mu=\mu_{\mathcal L}$. The measures $\mu(\ttarc{I_sI_t})$ and $\mu(\ssarc{I_sI_t})$ are computed from the intersection numbers $\alpha_r$ and $\beta_s$ as follows \cite{YG19}. For $r,r'$ both odd (respectively both even), let $\alpha_{r,r'}$ be the number of components of $L\cap D^U$ (respectively $L\cap D^D$) which cross both $\alpha_r$ and $\alpha_{r'}$. Then $\alpha_{r,r}=\alpha_r$ and $\alpha_{r',r}=\alpha_{r,r'}$. For $1\le i\le n-1$ and $r\in\{2i-1,2i\}$,
\begin{equation}\label{eq:alphapair}
\alpha_{r,r+2}=\min\bigl(\alpha_r-b_{i-1}^+,\ \alpha_{r+2}-(-b_i)^+\bigr),
\end{equation}
and for $r<r'$ of the same parity,
\begin{equation}\label{eq:alphamin}
\alpha_{r,r'}=\min\bigl(\alpha_{r,r+2},\alpha_{r+2,r+4},\dots,\alpha_{r'-2,r'}\bigr).
\end{equation}

Let $x_s\alpha_r$ be the number of components of $L$ in the same half disk as $\alpha_r$ which have one endpoint in $I_s$ and cross $\alpha_r$. Let $\varepsilon=0$ if $r$ is odd and $\varepsilon=1$ if $r$ is even. Then
\begin{equation}\label{eq:xs}
x_0\alpha_r=\alpha_{1+\varepsilon,\,r},\qquad x_n\alpha_r=\alpha_{2n-1+\varepsilon,\,r},\qquad
x_s\alpha_r=\bigl|\alpha_{2s+1+\varepsilon,\,r}-\alpha_{2s-1+\varepsilon,\,r}\bigr|
\end{equation}
for $1\le s\le n-1$, and $x_s\alpha_r=0$ if $r>2n$. Finally, for $0\le s<t\le n$,
\begin{equation}\label{eq:mu}
\mu(\ttarc{I_sI_t})=x_s\alpha_{2t-1}-x_s\alpha_{2t+1},\qquad \mu(\ssarc{I_sI_t})=x_s\alpha_{2t}-x_s\alpha_{2t+2}.
\end{equation}

Observe that these formulae decide from the Dynnikov coordinates of a curve $\gamma$ whether $\gamma$ lies in a witness set, since by Remark~\ref{rem:witnessdef}
\[
\gamma\in Q_{i,j}\quad\Longleftrightarrow\quad
\mu_\gamma(\ttarc{I_{i-1}I_j})>0\quad\text{and}\quad\mu_\gamma(\ssarc{I_{i-1}I_j})>0 .
\]

The following example illustrates the transition from Dynnikov coordinates to $\pi_1$-train track coordinates.

\begin{example}\label{ex:dynnikov}
Let $\gamma=t_{3,6}(c_{1,5})$ be the curve in $D_7$ of Example~\ref{ex:twist}, shown in Figure~\ref{fig:example}(b). Its Dynnikov coordinates are given by
\[
(a;b)=(0,1,1,1,0;\,0,-1,0,0,2).
\]
By \eqref{eq:beta} and \eqref{eq:alpha} we have
\[
(\beta_1,\dots,\beta_6)=(2,2,4,4,4,0)\qquad\text{and}\qquad(\alpha_1,\dots,\alpha_{14})=(1,1,1,1,1,3,1,3,1,3,2,2,0,0).
\]
We compute, for instance, $\mu_\gamma(\ttarc{I_2I_6})$ and $\mu_\gamma(\ssarc{I_2I_6})$. For brevity we write $\mu$ for $\mu_\gamma$. With $b_0=-\beta_1/2=-1$ and $b_6=\beta_6/2=0$, \eqref{eq:alphapair} gives
\[
(\alpha_{1,3},\alpha_{3,5},\dots,\alpha_{11,13})=(1,0,1,1,1,0),\qquad
(\alpha_{2,4},\alpha_{4,6},\dots,\alpha_{12,14})=(1,1,3,3,2,0).
\]
Then \eqref{eq:alphamin} gives that $\alpha_{3,11}=0$, $\alpha_{5,11}=1$, $\alpha_{3,13}=\alpha_{5,13}=0$, $\alpha_{4,12}=1$, $\alpha_{6,12}=2$ and $\alpha_{4,14}=\alpha_{6,14}=0$. By \eqref{eq:xs} and \eqref{eq:mu}, with $s=2$ and $t=6$,
\[
\mu(\ttarc{I_2I_6})=x_2\alpha_{11}-x_2\alpha_{13}=|\alpha_{5,11}-\alpha_{3,11}|-|\alpha_{5,13}-\alpha_{3,13}|=1-0=1,
\]
\[
\mu(\ssarc{I_2I_6})=x_2\alpha_{12}-x_2\alpha_{14}=|\alpha_{6,12}-\alpha_{4,12}|-|\alpha_{6,14}-\alpha_{4,14}|=1-0=1.
\]
Similarly one finds
\[
\mu(\ttarc{I_0I_2})=\mu(\ttarc{I_2I_6})=\mu(\ttarc{I_5I_6})=1,\qquad \mu(\ssarc{I_0I_6})=\mu(\ssarc{I_2I_5})=\mu(\ssarc{I_2I_6})=1,
\]
and all the other measures are $0$, in agreement with Figure~\ref{fig:pi1tt}. Furthermore, since $\mu(\ttarc{I_2I_6})>0$ and $\mu(\ssarc{I_2I_6})>0$, we can detect from the Dynnikov coordinates of $\gamma$ that $\gamma\in Q_{3,6}$.
\end{example}

\subsection{From $\pi_1$-train track coordinates to intersection numbers}\label{sec:fromarcs}
Conversely, the intersection numbers $\alpha_r$ and $\beta_s$, and hence the Dynnikov coordinates, are obtained from the measures $\mu(\ttarc{I_sI_t})$ and $\mu(\ssarc{I_sI_t})$ as follows \cite{YG19}. For $1\le i\le n$ and $1\le j\le n-1$,
\[
\alpha_{2i-1}=\sum_{s<i\le t}\mu(\ttarc{I_sI_t}),\qquad \alpha_{2i}=\sum_{s<i\le t}\mu(\ssarc{I_sI_t}),\qquad
\beta_j=\max\bigl(\alpha_{2j-1}+\alpha_{2j+2},\ \alpha_{2j}+\alpha_{2j+1}\bigr)-|L\cap I_j|,
\]
where $|L\cap I_j|=\sum_{t\ne j}\mu(\ttarc{I_jI_t})=\sum_{t\ne j}\mu(\ssarc{I_jI_t})$ is the number of crossings of $L$ with $I_j$.

\section*{Acknowledgements}
We thank Y\i{}ld\i{}ray Ozan and Mustafa Korkmaz for their comments on earlier drafts of this paper. This work was carried out as part of TÜBİTAK project 123F221.

\end{document}